\documentclass[reqno, 11pt]{amsart}
\title[Annealed invariance principle for random walks on random graphs in $\RR^d$]{Annealed invariance principle for random walks on random graphs generated by point processes in $\RR^d$}
\author[A. Rousselle]{Arnaud Rousselle$^\ddagger$
}
\thanks{
$\ddagger$ Normandie Universit\'e, Universit\'e de Rouen, 
Laboratoire de Math\'ematiques Rapha\"el Salem,
CNRS, UMR 6085, Avenue de l'universit\'e, BP 12, 
76801 Saint-Etienne du Rouvray Cedex, 
France
\newline
and 
\newline
Laboratoire Modal'X, Universit\'e Paris Ouest Nanterre La D\'efense, 200 avenue de la R\'epublique, 92000 Nanterre,
France
\newline
E-mail address: arnaud.rousselle@u-paris10.fr}
\usepackage{amssymb,xspace}
\usepackage{stmaryrd}
\usepackage{latexsym}
\usepackage{mathtools}
\usepackage{epsfig}
\usepackage{epsf}
\usepackage{float}
\usepackage{a4wide}
\usepackage{hyperref}
\usepackage{latexsym,amssymb,amsmath,epsfig,psfrag}
\usepackage{alltt}
\usepackage{cite}
\usepackage{color}
\usepackage[top=2.5cm,left=2cm,right=2cm,bottom=2.5cm]{geometry}

\def\0{\mathbf{0}}
\def\x{\mathbf{x}}
\def\d{\mathrm{d}}

\def\z{\mathbf{z}}
\def\N{\mathcal{N}}

\def\P{\mathcal{P}}

\def\E{\mathcal{E}}

\def\Q{\mathcal{Q}}

\def\X{\widetilde{X}}
\def\xit{\wideltilde{\xi}}

\def\x{\mathbf{x}}
\def\y{\mathbf{y}}
\def\z{\mathbf{z}}

\def\PP{\mathbb{P}}

\def\RR{\mathbb{R}}
\def\XX{\mathbb{X}}
\def\YY{\mathbb{Y}}

\def\ZZ{\mathbb{Z}}

\def\NN{\mathbb{N}}

\def\L{\mathcal{L}}
\def\gm{\mathfrak{m}}
\def\gcb{\overline{\mathfrak{c}}}
\def\Pt{\widetilde{P}}

\def\Nz{{\mathcal{N}_0}}
\def\Pz{{\mathcal{P}_0}}
\def\xiz{{\xi^0}}
\def\oC{{\overset{\circ}{C}}}
\def\gV{{\mathfrak{V}}}
\def\gVb{\overline{\mathfrak{V}}}
\def\gE{{\mathfrak{E}}}
\def\gEb{\overline{\mathfrak{E}}}
\def\gQ{{\mathfrak{Q}}}

\def\gG{{\mathfrak{G}}}
\def\gp{{\mathfrak{p}}}
\def\gGb{\overline{\mathfrak{G}}}
\def\xit{{\widetilde{\xi}}}
\def\gc{\mathfrak{c}}
\def\bPt{\widetilde{\mathbf{P}}}
\def\bEt{\widetilde{\mathbf{E}}}
\def\II{\mathbf{1}}
\newcommand{\sfrac}[2]{\kern.1em
        \raise.5ex\hbox{$#1$}\kern-.1em
        /\kern-.15em\lower.25ex\hbox{$#2$}}
\newtheorem*{ass*}{Assumptions}
\newtheorem{defi}{Definition}
\newtheorem{lemm}[defi]{Lemma}

\newtheorem{prop}[defi]{Proposition}

\newtheorem{theo}[defi]{Theorem}
\newtheorem{exem}[defi]{Example}
\newtheorem{rem}[defi]{Remark}

\newenvironment{dem}{\noindent {\bf Proof.}}
                    {\hfill $\square$}

\newcommand{\const}[1]{c_{#1}}

\def\bex{\begin{exem} \em }
\def\eex{\end{exem} }
\def\brem{\begin{rem} \em}
\def\erem{\end{rem} }
\def\bP{\mathbf{P}}
\def\bE{\mathbf{E}}
\def\bQ{\mathbf{Q}}
\def\deg{\operatorname{deg}}
\def\DT{\operatorname{DT}}
\def\Gab{\operatorname{Gab}}

\def\Vol{\operatorname{Vol}}
\begin{document}
\begin{abstract}
We consider simple random walks on random graphs embedded in $\mathbb{R}^d$ and generated by point processes such as Delaunay triangulations, Gabriel graphs and the creek-crossing graphs. Under suitable assumptions on the point process, we show an annealed invariance principle for these random walks. These results hold for a large variety of point processes including Poisson point processes,  Mat\'ern cluster and Mat\'ern hardcore processes which have respectively clustering and repulsiveness properties. The proof relies on the use the process of the environment seen from the particle. It allows to reconstruct the original process as an additive functional of a Markovian process under the annealed measure.  
\end{abstract}
\date{\today}
\maketitle



\maketitle

\section{Introduction and results}
\subsection{Introduction}
 Random walks in random environments are natural objects for studying flows, molecular diffusions, heat conduction or other problems from statistical mechanics. Many papers concerning random walks on random perturbations of the lattice $\ZZ^d$ or random subgraphs of $\ZZ^d$ were published during the last forty years. Particular examples include random walks on supercritical percolation clusters and the so-called random conductance model (see \cite{BiskupLN} and references therein). An annealed invariance principle for such random walks was obtained in \cite{DFGW84,DFGW} while its quenched counterpart is proved in \cite{SidoravS,BergerBiskup,BP,MP,MathieuQuenched}. Few papers were written in the case where the random environment is generated by a point process in $\RR^d$. In the last decade, techniques were developed to study random walks on complete graphs generated by point processes with jump probability which is a decreasing function of the distance between points possibly modified by the presence of an energy mark at each point. The main results on this model include an annealed invariance principle \cite{FSS,FM}, isoperimetric inequalities \cite{CF}, recurrence or transience results \cite{CFG}, a quenched invariance principle and heat kernel estimates \cite{CFP}.
In \cite{RTCRWRG}, we answer the question of recurrence or transience of random walks on random graphs constructed according to the geometry of a random (infinite and locally finite) set of points such that Delaunay triangulations or Gabriel graphs. In the present paper, we are interested in stating annealed invariance principles for simple random walks on random graphs generated by  point processes. More precisely, given a realization $\xi$ of a point process  in $\RR^d$, a connected graph $G(\xi)=(\xi, E_{G(\xi)})$, in which each vertex has finite degree, is constructed. While its vertex set is $\xi$ itself, its edge set $E_{G(\xi)}$ is obtained using deterministic rules based on the geometry of $\xi$. For instance, we will consider \emph{Delaunay triangulations}, \emph{Gabriel graphs} or \emph{creek-crossing graphs} generated by point processes. Their constructions will be described below. We then consider continuous-time random walks on $G(\xi)$ and investigate their annealed (or averaged) scaling limits.

In the sequel,  we will write $\N$ (resp. $\N_0$) for the collection of all infinite locally finite subsets of $\RR^d$ (resp. infinite locally finite subsets of $\RR^d$ containing the origin). We denote by $\P$ (resp. $\P_0$)  the law of a stationary, ergodic and simple point process (resp. its Palm version). We recall that $\P_0(\N_0)=1$ (see \cite[\S 3.3]{SW}).  We denote by $\E$ (resp. $\E_0$) the expectation with respect to $\P$ (resp. $\P_0$).  We assume that $\P$-almost all $\xi$ is aperiodic, {\it i.e.\@\xspace}, with $\P$-probability 1:
\[\xi\neq \tau_x\xi,\,\forall x\in \RR^d,\]
where $\tau_x$ is the translation $\tau_x\xi=\sum_{y\in\xi}\delta_{y-x}$. This implies that the same holds for $\P_0$-almost all $\xiz$.

 For almost any realization of $\xi^0$, we define a variable speed random walk (VSRW) on $G(\xi^0)$ as follows. Let $D(\RR_+;\xi^0)$ be the space of c\`adl\`ag $\xi^0$-valued functions on $\RR_+$ endowed with the Skorohod topology, and $X^{\xi^0}_t=X^{\xi^0}(t)$, $t\geq 0$ be the coordinate map from $D(\RR_+;\xi^0)$ to $\xi^0$. For $x\in\xi^0$, we write $P^{\xi^0}_x$ for the probability measure on $\Omega_{\xi^0}:=D(\RR_+;\xi^0)$ under which the coordinate process $(X^{\xi^0}_t)_{t\geq 0}$ is the continuous-time Markov process starting from $X^{\xi^0}_0=x$ with infinitesimal generator given by:
\begin{equation}
\L^{\xi^0}f(y):=\sum_{z\in \xi^0}c^{\xi^0}_{y,z}\big(f(z)-f(y)\big),\, y\in \xi^0,\label{DefGeneChapAIP}
\end{equation}
where $c^{\xi^0}_{y,z}$ is equal to the indicator function of the event `$\{y,z\}\in E_{G(\xi^0)}$'.
We denote by $E^{\xi^0}_x$ the expectation with respect to $P^{\xi^0}_x$. Note that $P_x^\xiz$ can be seen as a probability measure on $\Omega:=D(\RR_+;\RR^d)$ concentrated on $\Omega_{\xi^0}$.

The behavior of $X^{\xi^0}_t$ under $P^{\xi^0}_x$ can be described as follows: starting from the point $x$, the particle waits for an exponential time of parameter $\operatorname{deg}_{G(\xi^0)}(x)$ and then chooses, uniformly at random, one of its neighbors in $G(\xi^0)$ and moves to it. This procedure is then iterated with independent  times (of parameter the degree of the current location in $G(\xi^0)$). The same arguments as in \cite[Appendix A]{FSS} show that these random walks are well-defined ({\it i.e.\@\xspace} that no explosion occurs) as soon as $\E_0[\deg_{G(\xi^0)}(0)]<\infty$.  

Let $P$ be the so-called annealed semi-direct product measure on $\N_0\times\Omega$:
\[P\big[F(\xi^0,X^{\xi^0}_\cdot )\big]=\int_{\N_0} P^{\xi^0}_0\big[F(\xi^0,X^{\xi^0}_\cdot )\big]\P_0(\d \xi^0).\]

Note that $(X^{\xi^0}_t)_{t\geq 0}$ is no longer Markovian under $P$. Using technics similar to those from \cite[Section 4]{DFGW} and \cite{FSS}, we will use the \emph{point of view of the particle} and \cite[Theorem 2.2]{DFGW} to obtain annealed invariance principles for random walks on specific examples of random geometric graphs with suitable underlying point processes. Let us describe the precise models.

\subsection{Conditions on the point process}\label{condprocAIP}

In all what follows we assume that $\P$ is the distribution of a simple, stationary point process $\xi$ satisfying $\E\big[\#(\xi\cap [0,1]^d)^8\big]<\infty$.
We also suppose that the second factorial moment measure of the point process $\xi$:
\[\Lambda^{(2)}(A_1\times A_2):=\E\left[\#\left\{(x_1,x_2): x_1\in\xi\cap A_1, x_2\in\xi\cap A_2\setminus\{x_1\}\right\}\right], \quad A_1,A_2\in \mathcal{B}(\RR^d),\]
is absolutely continuous, with bounded density, with respect to the $2d$-Lebesgue measure. We assume that $\P$ is aperiodic, isotropic ({\it i.e.\@\xspace} its distribution is invariant under rotations around the origin) and almost surely in general position ({\it i.e.\@\xspace} there are no $d+1$ points (resp. $d+2$ points) in a $(d-1)$-dimensional affine subspace (resp. in a sphere), see \cite{Zessin}). Using {\it e.g.\@\xspace} \cite[Lemma B.2]{CFP}, one can see that the Palm version $\P_0$ of $\P$ also satisfies these properties. It is moreover assumed that $\xi$ has a \emph{finite range of dependence $k$}, {\it i.e.\@\xspace}, for any disjoint Borel sets $A, B\subset \RR^d$ with $\mathrm{d}(A,B):=\inf\{\Vert x-y\Vert\,:\,x\in A,y\in B\}\geq k$, $\xi\cap A$ and $\xi\cap B$ are independent. In particular, the point process is ergodic. 

Finally, the following hypothesis on the void probabilities {\bf (V)} and the deviation probabilities {\bf (D)} are needed:

\begin{itemize}
\item[{\bf (V)}]
there exists a positive constant $\const{1}$ such that for $L$ large enough:
\begin{equation*}
\P\left[\#\left(\xi\cap [-L/2,L/2]^d\right)=0\right]\leq e^{-\const{1}L^d};
\end{equation*}
\item[{\bf (D)}] \label{condprocAIP2} 
there exist positive constants $\const{2},\const{3}$ such that for $L$ large enough:
\begin{equation*}
\P\left[\#\left(\xi \cap [-L/2,L/2]^d\right)\geq \const{2}L^d\right]\leq e^{-\const{3}L^d}.
\end{equation*}
\end{itemize}

Assumptions {\bf (V)} and {\bf (D)} imply the following analogue estimates for the Palm version of the point process (see Lemma \ref{LemmeUtilPoisson} (\ref{VimpliesV'}) and (\ref{DimpliesD'})):  
\begin{itemize}
\item[{\bf (V')}]
there exists a positive constant $\const{4}$ such that for $L$ large enough:
\begin{equation*}\label{MomentsExpoVoid}
\P_0\left[\#\left(\xiz\cap \left(x+[-L/2,L/2]^d\right)\right)=0\right]\leq e^{-\const{4}L^d},\quad \forall x\in \RR^d;
\end{equation*}
\item[{\bf (D')}] \label{condprocAIP3} 
there exist positive constants $\const{5},\const{6}$ such that for $L$ large enough:
\begin{equation*}
\P_0\left[\#\left(\xiz\cap \left(x+[-L/2,L/2]^d\right)\right)\geq \const{5}L^d\right]\leq e^{-\const{6}L^d},\quad \forall x\in \RR^d.
\end{equation*}
\end{itemize}

Let us note that, due to the stationarity of the point process and {\bf (D)}, we have:
\[\lim_{N\rightarrow \infty}N^8\P\big[\#\big(\xi\cap [-N,N]^d)\geq \const{7} N^d\big]=0 \mbox{ for some }\const{7}>0.\] 

These assumptions are in particular satisfied if $\xi$ is:
\begin{itemize}
\item a homogeneous Poisson point process (PPP),
\item a Mat\'ern cluster process (MCP),
\item a Mat\'ern hardcore process I or II (MHP I/II).
\end{itemize}

We refer to Appendix \ref{Expp} for an overview on each of these point processes. Note that these processes have different interaction properties: for PPPs there is no interaction between points, MCPs exhibit clustering effects whereas points in MHPs repel each other. 
\subsection{The graph structures}
Given a realization $\xi$ of the point process satisfying the previous assumptions, we write $$\mathrm{Vor}_\xi(\x):=\{x\in\RR^d\, :\,\Vert x-\x\Vert\leq\Vert x-\y\Vert,\forall\y\in\xi \}$$ for the Voronoi cell of $\x\in\xi$; $\x$ is called the nucleus or the seed of the cell. The Voronoi diagram of $\xi$ is the collection of the Voronoi cells. It tessellates $\RR^d$ into convex polyhedra. See \cite{Moller,NewPerspectives} for an overview on these tessellations.

The graphs considered in the sequel are:  
\begin{description}
\item[{\bf $\operatorname{DT}(\xi)$}] the \emph{Delaunay triangulation} of $\xi$. It is the dual graph of its Voronoi tiling. It has $\xi$ as vertex set and there is an edge between $\x$ and $\y$ in $\operatorname{DT}(\xi)$ if $\mathrm{Vor}_\xi(\x)$ and $\mathrm{Vor}_\xi(\y)$ share a $(d-1)$-dimensional face. Another useful characterization of $\operatorname{DT}(\xi)$ is the following: a simplex $\Delta$ is a cell of $\operatorname{DT}(\xi)$ \emph{iff} its circumscribed sphere has no point of $\xi$ in its interior. Note that this triangulation is well defined since $\xi$ is assumed to be in general position. 

These graphs are  widely used in many fields such as astrophysics \cite{Ramella}, cellular biology \cite{Poupon}, ecology \cite{Roque} and telecommunications \cite{BB3}. 
\item[ {\bf $\operatorname{G_n}(\xi)$, $\operatorname{n}\geq 2$}] the \emph{creek-crossing graphs} generated by $\xi$. These are sub-graphs of $\DT (\xi)$. Each $\operatorname{G_n}(\xi)$ has $\xi$ as vertex set and two vertices $\x,\y \in \xi$ are connected by an edge in $\operatorname{G_n}(\xi)$ \emph{iff} there do not exist an integer $k\leq \operatorname{n}$ and vertices $u_0=\x, \dots, u_k=\y\in\xi$ such that $\Vert u_i-u_{i+1}\Vert <\Vert \x-\y\Vert$ for all $i\in\{0, \dots, k-1\}$. Let us recall that  $\operatorname{G_2}(\xi)$ is also known as the \emph{relative neighborhood graph} generated by $\xi$ and that the intersection of the  $\operatorname{G_n}(\xi)$ (in the graph sense) is the Euclidean minimum spanning forest of $\xi$ (see \cite{HNS12,HNGS}).
\item[{\bf $\operatorname{Gab}(\xi)$}] the \emph{Gabriel graph} of $\xi$. Its vertex set is $\xi$ and there is an edge between $\x,\y\in \xi$ if the ball of diameter $[\x,\y]$ contains no point of $\xi$ in its interior. Note that $\operatorname{Gab}(\xi)$ is a subgraph of $\operatorname{DT}(\xi)$ and contains the Euclidean minimum spanning forest of $\xi$ and the relative neighborhood graph. It has for example applications in geography, routing strategies, biology or tumor growth analysis (see\cite{BBD,Gab1,Gab2}).
\end{description}

\subsection{Main results}
The aim of this paper is to prove the following theorem.
\begin{theo}\label{ThprincAIP}
Assume that $\P$ satisfies the conditions described in Subsection \ref{condprocAIP}. 

Let $(X^{\xi^0}_t)_{t\geq 0}$ be the random walk on $\DT (\xiz)$, $\Gab (\xiz)$ or $\operatorname{G_n} (\xiz), \,\operatorname{n}\geq 2$ with generator (\ref{DefGeneChapAIP}).

 Then, the rescaled process $\underline{Y}^{\xi^0,\varepsilon}:=(\varepsilon X^{\xi^0}_{\varepsilon^{-2}t})_{t\geq 0}$ converges weakly in $\P_0$-probability to a nondegenerate Brownian motion $\underline{W}^{\sigma^2\mathbf{I}_d}$ with:
\begin{equation} \label{EqSigma2}
 \sigma^2=\lim_{t\rightarrow\infty}\frac{1}{t}\E_0\big[E^{\xi^0}_0\big[(X^{\xi^0}_t\cdot e_1)^2\big]\big]>0,
 \end{equation}
 where $e_1$ is the first vector of the standard basis of $\RR^d$. In other words, for any bounded continuous function $F$ on the path space $D([0,,\infty),\RR^d)$ endowed with Skorohod topology
 \[\E_0^{\xiz}\left[F\left(\underline{Y}^{\xi^0,\varepsilon}\right)\right]\longrightarrow\mathbb{E} \left[F\left(\underline{W}^{\sigma^2\mathbf{I}_d}\right)\right]\quad\mbox{in }\P_0\mbox{-probability.}\]
 
 This is the case if the underlying point process is a PPP, a MCP or a MHP I/II.
\end{theo}

In the companion paper \cite{QIP}, we prove a quenched invariance principle in the particular case of the simple random walk on the Delaunay triangulation generated by suitable point processes (see also \cite{FGG} for the 2-dimensional case). The method used in \cite{QIP} relies notably on heat-kernel estimates and demands a precise control of regularity (isoperimetry, volume growth, ...) of the random environment. The main interest of the present paper is that it does not require such precise estimates. Consequently, it applies to a wider class of point processes and graphs. 

As mentioned above, we closely follow the method of \cite{FSS} and we want to apply the results of De Masi, Ferrari, Goldstein and Wick. A key idea of \cite{DFGW}, which also appears in the prior works by Kozlov \cite{Kozlov} and by Kipnis and Varadhan \cite{KV}, is to introduce an auxiliary Markov process: \emph{the environment seen from the particle}. Actually, thanks to the aperiodicity of the medium,  $(X^{\xi^0}_t)_{t\geq 0}$ can be reconstructed as a functional of the environment process. 

We will prove the following preliminary results stated in the general context of the introduction which imply Theorem \ref{ThprincAIP}. For $x,y\in\xiz$, we write $x\sim y$ if $\{x, y\}\in E_{G(\xi^0)}$. This definition depends on $\xi^0$ which is omitted from the notation.

\begin{prop}\label{PropCV}  Let $\P$ be a stationary aperiodic ergodic simple point process in $\RR^d$. Assume that $\max_{x\sim 0}\Vert x\Vert$ and $\deg_{G(\xiz)}(0)$ admit polynomial moments of any order under $\P_0$. 

 Then, the rescaled process $\underline{Y}^{\xi^0,\varepsilon}:=(\varepsilon X^{\xi^0}_{\varepsilon^{-2}t})_{t\geq 0}$ converges weakly in $\P_0$-probability to a Brownian motion $\underline{W}^{D}$ with covariance matrix $D$  characterized by equation (\ref{EqVaria1}).
\end{prop}
\begin{rem}
\begin{enumerate}
\item If the underlying point process is isotropic, then $D=\sigma^2\mathbf{I}_d$ where $\sigma^2$ is given by (\ref{EqSigma2}). 
\item In the applications presented here $\max_{x\sim 0}\Vert x\Vert$ and $\deg_{G(\xiz)}(0)$ actually admit exponential moments as shown in \cite[\S 11]{QIP}.
\end{enumerate}
\end{rem}

Note that the diffusion matrix $D$ of $\underline{W}^{D}$ can be degenerate. We will show that, under additional assumptions, it is not. To this end, we will need to discretize the space $\RR^d$ and to consider  boxes of side $K\geq 1$: \[B_\z=B^K_\z:=K\z+\Big[-\frac{K}{2},\frac{K}{2}\Big]^d,\,\z\in \ZZ^d.\] 
We will also define so-called `\emph{good boxes}' in the same spirit as in \cite{RTCRWRG} which allow us to deduce `nice properties' of $G(\xi)$ from `nice properties' of a suitable percolation process on $\ZZ^d$. 

\begin{prop}\label{PropNonDege} 
Let $\P$ be the law of a point process satisfying the assumptions of Proposition \ref{PropCV}. Assume in addition that $\E\big[\#(\xi\cap [0,1]^d)^8\big]<\infty$ and that there exists $\const{7}>0$ such that:
\begin{equation}\label{ConProcCV}
\lim_{N\rightarrow \infty}N^8\P\big[\#\big(\xi\cap [-N,N]^d)\leq \const{7}N^d\big]=0.
\end{equation} 

If one can associate to $\mathcal{P}$-a.a.\@\xspace $\xi$ a subset of the boxes $\{B_\z,\,\z\in \ZZ^d\}$, called \emph{ ($\xi$-)good boxes}, such that:
\begin{enumerate}
\item \label{CondNonDege1} in each good box $B_\z$, one can choose a reference vertex $v_\z\in B_\z\cap \xi$,
\item \label{CondNonDege2} there exists $L$ such that to each pair of neighboring good boxes $B_{\z_1}$ and $B_{\z_2}$, one can associate a path in $G(\xi)$, $(v_{\z_1}, \dots,v_{\z_2})$, between the respective reference vertices $v_{\z_1}$ and $v_{\z_2}$ of these boxes satisfying:  
\begin{enumerate}
\item $(v_{\z_1},  \dots ,v_{\z_2})\subset B_{\z_1}\cup B_{\z_2}$,
\item the length of $(v_{\z_1}, \dots,v_{\z_2})$ in the graph distance is bounded by $L$,
\end{enumerate}
\item \label{CondNonDege3} if $\xi$ is distributed according to $\P$, the process $\XX=\{X_\z, \, \z\in\ZZ^d\}:=\{\mathbf{1}_{B_\z {\text{ is }\xi\text{-good}}}, \, \z\in\ZZ^d\}$ stochastically dominates an independent Bernoulli site percolation process on $\ZZ^d$ with parameter $p\geq p^\text{site}_c(\ZZ^2)$,
\end{enumerate}
 then the limiting Brownian motion $\underline{W}^{D}$ in Proposition \ref{PropCV} has a nondegenerate covariance matrix.
\end{prop}

\subsection{Outline of the paper}
Sections \ref{Checkhyp} and \ref{Nondege} are respectively devoted to the proof of Propositions \ref{PropCV} and \ref{PropNonDege}. Theorem \ref{ThprincAIP} is derived from Propositions \ref{PropCV} and \ref{PropNonDege} in Section \ref{DerivTh1}. More precisely, arguments which allow to deduce polynomial moments for $\deg_{\DT(\xi^0)}(0)$ and $\max_{x\sim 0\text{ in }\DT(\xi^0)}\Vert x\Vert$ under $\P_0$ from the assumptions on the point process are given in Subsection \ref{AppMoments}. In Section \ref{SectGoodBoxes3}, we explain how to construct the `good boxes' of Proposition \ref{PropNonDege} in the case of creek-crossing graphs. The nondegeneracy of the limiting Brownian motion then follows in the other cases by Rayleigh monotonicity principle. Appendix \ref{Expp} contains the verification of the fact that PPPs, MCPs and MHPs satisfy the assumptions of Theorem \ref{ThprincAIP}.

\section{Proof of Proposition \ref{PropCV}}\label{Checkhyp}
We show how to use \cite[Theorem 2.2]{DFGW} in our framework. This result was used in the context of random walks in random conductance models (see \cite{DFGW84,DFGW}) and of random walks on complete graphs generated by point processes with jump rates depending on energy marks and point distances (see \cite{FSS}). It allows us to obtain the convergence of the rescaled process to a Brownian motion in Proposition \ref{PropCV} and gives an expression for the diffusion matrix of the limiting Brownian motion. The nondegeneracy of this matrix and the proof of Proposition \ref{PropNonDege} are postponed to Section \ref{Nondege}. For this section, we assume that the assumptions of Proposition \ref{PropCV} are satisfied.

\subsection{The point of view of the particle}\label{PointDeVueDeLaparticle} Let us denote by $(\tau_x)_{x\in \RR^d}$ the group of translations in $\RR^d$ which acts naturally on $\N$ as follows: $\tau_x\xi=\sum_{y\in\xi}\delta_{y-x}$. Note that, if $\xi^0\in\mathcal{N}_0$ and $x\in\xi^0$, then $\tau_x\xi^0\in\mathcal{N}_0$. The theory developed in \cite{KV,DFGW} makes extensive use of the process of the \emph{environment seen from the particle} $(\tau_{X^{\xi^0}_t}\xi^0)_{t\geq 0}$. Whereas the original random walk is defined on the probability space $(\Omega_{\xi^0},P^{\xi^0}_0)$, this auxiliary process is defined on $\Xi^0:=D(\RR_+;\N_0)$ and we write $\bP_{\xi^0}$ for its distribution. A generic element of $\Xi^0$ is denoted by $\underline{\xi^0}=(\xi^0_t)_{t\geq 0}$. The environment process can be seen from an `annealed' point of view, that is under the distribution $\bP:=\int\bP_{\xi^0}[\cdot]\P_0(\d \xi^0)$. In this setting, under the assumption of aperiodicity, this is the continuous-time Markov process with initial measure $\P_0$ and transition probabilities given by:
\begin{align*}
\bP\big[\xi^0_{s+t}=\xi'\big\vert \xi^0_s=\xi\big]&=\bP_\xi\big[\xi^0_t=\xi'\big]
:=\left\lbrace\begin{array}{ll}
P^\xi_0\big[X^\xi_t=x\big]&\mbox{if } \xi'=\tau_x\xi,\,x\in\xi\\
0 &\mbox{otherwise}
\end{array}\right. 
,\,\forall s,t\geq 0.
\end{align*}

Let us explain how the original walk can be reconstructed from the process $\underline{\xi^0}$. For $t\geq 0$, let $X_t: \Xi^0\longrightarrow \RR^d$ be the random variable defined by:
\begin{equation}
X_t(\underline{\xi^0}):=\sum_{\tiny
\begin{array}{c}
s\in [0,t]:\\ \Delta_s(\underline{\xi^0})\neq 0
\end{array}}\Delta_s(\underline{\xi^0}),\label{deffamcov}
\end{equation}  
where $\Delta_s(\underline{\xi^0})=y$ if there exists $y$ such that $\xi^0_s=\tau_y\xi^0_{s^-}$ and is 0 otherwise. Note that $\Delta_s$ is well defined for a.a.\@\xspace realization of the environment process since $\xi^0$ is almost surely aperiodic. The law of $(X_t(\underline{\xi^0}))_{t\geq 0}$ under $\bP_{\xi^0}$ is then equal to $P_0^{\xi^0}$, the distribution of the original random walk starting at $0$. In particular, the diffusion matrix can be calculated using the following equality:
\begin{equation*}
\E_0\big[E_0^{\xi^0}\big[(X_t^{\xi^0}\cdot a)^2\big]\big]=\bE\big[(X_t\cdot a)^2\big].
\end{equation*}

The family $\{X_{[s,t]}:=X_t-X_s\}_{t>s\geq 0}$ is an antisymmetric additive covariant family of random variables in the sense of \cite{DFGW}, and $X_t$ has c\`adl\`ag paths. Hence, to apply \cite[Theorem 2.2]{DFGW}, we only need to check that:
\begin{enumerate}
\item \label{hyp1} the environment process $(\xi^0_t)_{t\geq 0}$ is reversible and ergodic w.r.t.\@\xspace $\bP$;
\item \label{hyp2} the random variables $X_{[s,t]},\,t>s\geq 0$ are in $L^1(\Xi^0,\bP)$;
\item \label{hyp3} the mean forward velocity:
\[\varphi (\xi^0)=\lim_{t\downarrow 0}\frac{1}{t}\bE_{\xi^0}[X_t]\]
exists as a strong limit in $L^1(\N_0,\P_0)$; 
\item \label{hyp4} the martingale $M_t:=X_t-\int_0^t \varphi (\xi^0_s)\d s$ is square integrable under $\bP$.
\end{enumerate} 

The rest of this section is devoted to the verification of these hypotheses.
\subsection{Underlying discrete-time Markov chains and description of the dynamics}\label{ChaineIncluse}
It is convenient to decompose the dynamics in terms of the jump chain and the jump times as described in this subsection.
Given $\xi^0\in \N_0$, let us write $\widetilde{\Omega}_{\xi^0}:=(\xi^0)^\NN$ and denote by $(\widetilde{X}^{\xi^0}_n)_{n\in \NN}$ a generic element of $\widetilde{\Omega}_{\xi^0}$. For $x\in \xi^0$, let $\widetilde{P}_x^{\xi^0}$ be the distribution on $\widetilde{\Omega}_{\xi^0}$ of the discrete-time simple nearest neighbor random walk on $G(\xi^0)$. We denote by $\widetilde{E}_x^{\xi^0}$ the corresponding expectation. On another probability space $(\Theta , \bQ)$, let $T^{\xi^0}_{n,x},\,x\in\xi^0,\,n\in \NN$ be independent exponential random variables with mean $1/\operatorname{deg}_{G(\xi^0)}(x)$, and define on $(\widetilde{\Omega}_{\xi^0}\times \Theta,\widetilde{P}_x^{\xi^0}\otimes\bQ)$ the random variables:
\[R^{\xi^0}_0:=0;\,R^{\xi^0}_n:=\sum_{i=0}^{n-1}T^{\xi^0}_{i,\widetilde{X}^{\xi^0}_i}, \,n\geq 1;\]
\[n^{\xi^0}_*(t):=n\mbox{ where }R^{\xi^0}_n\leq t<R^{\xi^0}_{n+1}.\]
These random variables can be thought respectively as the instant of $n$th jump and the number of jumps up to time $t$ for the continuous-time random walk with environment $\xi^0$ starting at the origin.
The same arguments as in \cite[Appendix A]{FSS} show that almost surely no explosion occurs, {\it i.e.}
\begin{equation*}\label{EqWellDefined}
\widetilde{P}^{\xi^0}_x\otimes \bQ \big[\lim_{n\uparrow \infty}R^{\xi^0}_n=\infty\big]=1.
\end{equation*}
In particular, for a.a.\@\xspace $\xi^0$, $n^{\xi^0}_*(t)$ is well defined for all $t\geq 0$ and $(\widetilde{X}^{\xi^0}_{n^{\xi^0}_*(t)})_{t\geq 0}$ is a jump process which satisfies the same dynamics (under $\widetilde{P}^{\xi^0}_x\otimes \bQ$) as $(X^{\xi^0}_t)_{t\geq 0}$ (under $P^{\xi^0}_x$): starting from the point $x$, the particle waits for an exponential time of parameter $\operatorname{deg}_{G(\xi^0)}(x)$, chooses uniformly  one of its neighbors in $G(\xi^0)$ and moves to it. 

It is useful to introduce the probability measure $\Q_0$ on $\N_0$ given by:
\begin{equation*}
\Q_0(\d \xi^0)=\frac{\operatorname{deg}_{G(\xi^0)}(0)}{\E_0\big[\operatorname{deg}_{G(\xi^0)}(0)\big]}\P_0(\d \xi^0),
\end{equation*}
and $\widetilde{P}:=\int \widetilde{P}_0^{\xi^0}\Q_0 (\d \xi^0)$, the annealed law of the (original) discrete-time random walk. In the same way as in Subsection \ref{Reversibility}, one can obtain reversibility and ergodicity of the discrete-time process under $\widetilde{P}$. Let $\widetilde{E}$ denote the expectation under $\widetilde{P}$.  

One can also consider the discrete-time version of the environment process $(\widetilde{\xi}^0_n)_{n\in\NN}:=(\tau_{\widetilde{X}^{\xi^0}_n}\xi^0)_{n\in\NN}$ and write $\widetilde{\bP}_{\xi^0}$ for its distribution on the path space $\widetilde{\Xi}^0:=\N_0^\NN$. When $\xi^0$ is aperiodic the same identification as in Subsection \ref{PointDeVueDeLaparticle} is possible. It is easy to see that the processes $(\operatorname{deg}_{G(\xi^0)}(\widetilde{X}^{\xi^0}_n))_{n\in\NN}$ on $(\widetilde{\Omega}_{\xi^0},\widetilde{P}_0^{\xi^0})$ and $(\operatorname{deg}_{G(\widetilde{\xi}^0_n)}(0))_{n\in\NN}$ on $(\widetilde{\Xi}^0,\widetilde{\bP}_{\xi^0})$ have the same law. One can also verify that the process $(Y_n)_{n\in \NN}$ defined by $Y_0=0$ and:
\[Y_n:=\sum_{i=0}^{n-1}\Delta(\widetilde{\xi}^0_i,\widetilde{\xi}^0_{i+1}),\]
with $\Delta(\widetilde{\xi}^0_i,\widetilde{\xi}^0_{i+1})$ defined in the same way as in (\ref{deffamcov}), has distribution $P^{\xi^0}_0$ on $\widetilde{\Omega}_{\xi^0}$. If there does not exist $x\in\xi$ such that $\xi'=\tau_x\xi$, we set $\Delta (\xi,\xi'):=0$.

\begin{rem}
By slight modifications of the arguments given in Subsection \ref{Reversibility}, one can obtain reversibility and ergodicity of $(\widetilde{\xi}^0_n)_{n\in\NN}$ under $\widetilde{\bP}:=\int \widetilde{\bP}_{\xi^0}\Q_0 (\d \xi^0)$.  This can be used to derive an annealed invariance principle for the discrete-time random walk thanks to \cite[Theorem 2.1]{DFGW}.  
\end{rem}

\subsection{Reversibility and ergodicity}\label{Reversibility}
\subsubsection{Reversibility} We want to show that the environment process $(\xi_t^0)_{t\geq 0}$ is reversible with respect to $\bP$ and we follow the arguments of the proof of \cite[Proposition 2]{FSS}. Using the description of the dynamics given in Subsection \ref{ChaineIncluse}, one can check that, for any $n\in\NN$ and any $\xi_0=\xiz, \xi_1,\dots,\xi_n=\xiz ' \in\N_0$:
\[\bP_\xiz\Big[n_*^\xiz(t)=n, \xit^0_{R_1^\xiz}=\xi_1,\dots,\xit^0_{R_n^\xiz}=\xi_n\Big]=\bP_{\xiz'}\Big[n_*^\xiz(t)=n, \xit^0_{R_1^\xiz}=\xi_{n-1},\dots,\xit^0_{R_n^\xiz}=\xi_0\Big],\]
and deduce that:
\[\bP_\xiz\big[\xi^0_t=\xiz'\big]=\bP_{\xiz'}\big[\xi^0_t=\xiz\big],\quad t>0.\]
Hence, for any positive measurable functions $f,g$ and any $t>0$, thanks to Neveu's exchange formula (see \cite[Theorem 3.4.5]{SW}), we obtain:
\begin{align*}
\bE\big[f(\xi^0_0)g(\xi^0_t)\big]&=\int_{\N_0}\sum_{x\in\xiz}f(\xi^0)g(\tau_x\xi^0)\bP_\xiz\big[\xi^0_t=\tau_x\xiz\big]\P_0(\d\xiz)\\
&=\int_{\N_0}\sum_{x\in\xiz}f(\tau_{x}\xi^0)g(\xi^0)\bP_{\tau_{x}\xiz}\big[\xi^0_t=\xiz\big]\P_0(\d\xiz)\\
&=\int_{\N_0}\sum_{x\in\xiz}f(\tau_{x}\xi^0)g(\xi^0)\bP_{\xiz}\big[\xi^0_t=\tau_x\xiz\big]\P_0(\d\xiz)\\
&=\bE\big[f(\xi^0_t)g(\xi^0_0)\big].
\end{align*}
The reversibility of $(\xi_t^0)_{t\geq 0}$ then follows.
\subsubsection{Ergodicity} We follow the proof of \cite[Proposition 2]{FSS}. Note that it suffices to ascertain that for any measurable $A\subset \N_0$ satisfying $\bP_{\xi^0}[\xi^0_t\in A]=\mathbf{1}_A(\xi^0)$ for a.a.\@\xspace $\xi^0$, we have $\P_0(A)=0\mbox{ or }1$. Given such an $A$, there is a subset $B\subset A$ such that $\P_0[A\setminus B]=0$ and $\bP_{\xi^0}[\xi^0_t\in A]=1$ for all $\xi^0\in B$. For $x\in\xi^0$, $\xiz\in B$, one then has $\bP_{\xi^0}[\xi^0_t=\tau_x\xi^0\in A]=\bP_{\xi^0}[\xi^0_t=\tau_x\xi^0]>0$ by the positivity of the jump rates. Thanks to \cite[Lemma 1 (iii)]{FSS}, $\P_0(A)=0\mbox{ or }1$ which completes the proof.

\subsection{Integrability of $\mathbf{X_{[s,t]}=X_t-X_s,\, t>s\geq 0}$}\label{Integrability} This is an immediate consequence of:
\begin{lemm}\label{lemmMoments}
Let $\gamma\in\NN^*=\{1, 2, \dots\}$ and $ t_0>0$. There exists a constant $0< \const{8}=\const{8}(\gamma,t_0)<\infty$ such that:
\begin{equation}\label{EqRemIntegrability}
\E_0\big[E_0^{\xi^0}\big[\Vert X^{\xi^0}_t\Vert^\gamma\big]\big]\leq \const{8}\sqrt{t}, \quad\mbox{for all }t\leq t_0.
\end{equation}
\end{lemm}
To obtain this lemma the reader should redo the calculations of the proof of \cite[Proposition 1]{FSS} with $p=q=2$ using that $\deg_{G(\xi^0)}(0)$ and $\max_{x\sim 0}\Vert x\Vert$ admit polynomial moments and make the dependence on $t$ apparent. 
\subsection{Existence of the mean forward velocity}\label{SectMeanVelocity}
We want to verify that the mean forward velocity:
\[\lim_{t\downarrow 0}\frac{1}{t}\bE_{\xi^0}[X_t]\]
exists as a strong limit in $L^1$. Actually, we prove the following stronger result which is also useful for the square integrability of the martingale $M_t$ in Subsection \ref{MartL2}:
\begin{lemm}\label{lemmvelocity} 
Let $\varphi: \N_0\longrightarrow \RR^d$ defined by:
\begin{equation}\label{defphi}
\varphi(\xiz):=\sum_{x\sim 0} x.
\end{equation}
Then,
\begin{equation*}
\frac{1}{t}\bE_{\cdot}[X_t]\xrightarrow[t\rightarrow 0]{\, L^2(\N_0,\P_0)\,}
  \varphi.
\end{equation*}
\end{lemm}
\begin{dem}  As in Subsection \ref{ChaineIncluse}, we denote by $n^{\xi^0}_*(t)$ the number of jumps of the environment process starting at $\xiz$ up to time $t$. First, we use that $X_t\mathbf{1}_{n^{\xi^0}_*(t)=0}=X_0=0$ and $(a+b)^2\leq 2(a^2+b^2)$ to write:
\begin{align}
\frac{1}{t^2}\int_{\N_0}\big\Vert &\bE_\xiz\big[X_t\big]-t\varphi (\xiz)\big\Vert^2\P_0(\d \xiz)\nonumber\\
=&\frac{1}{t^2}\int_{\N_0}\big\Vert \bE_\xiz\big[X_t\mathbf{1}_{n^{\xi^0}_*(t)=1}\big]-t\varphi (\xiz)+\bE_\xiz\big[X_t\mathbf{1}_{n^{\xi^0}_*(t)\geq 2}\big]\big\Vert^2\P_0(\d \xiz)\nonumber\\
\leq &\frac{2}{t^2}\int_{\N_0}\big\Vert \bE_\xiz\big[X_t\mathbf{1}_{n^{\xi^0}_*(t)=1}\big]-t\varphi (\xiz)\big\Vert^2\P_0(\d \xiz)+\frac{2}{t^2}\int_{\N_0}\big\Vert \bE_\xiz\big[X_t\mathbf{1}_{n^{\xi^0}_*(t)\geq 2}\big]\big\Vert^2\P_0(\d \xiz).\label{EqVelo2}
\end{align}
We need to show that the two summands in (\ref{EqVelo2}) go to 0 with $t$. We first deal with the first term in the r.h.s.\@\xspace of (\ref{EqVelo2}).

Note that, with the notation of Subsection \ref{ChaineIncluse}, for $x\in\xiz$:
\begin{align*}
\bP_\xiz\big[\xi^0_t=\tau_x\xiz,n_*^\xiz(t)=1\big]&\leq\Pt_0^\xiz\otimes\bQ\big[\X_1^\xiz=x ,\,T_{0,0}^\xiz\leq t\big]\nonumber\\
&=\Pt_0^\xiz\big[\X_1^\xiz =x\big]\big(1-e^{-\deg_{G(\xiz)}(0)t}\big)\nonumber\\
&\leq \frac{\mathbf{1}_{x\sim 0}}{\deg_{G(\xiz)}(0)}\deg_{G(\xiz)}(0)t=\mathbf{1}_{x\sim 0}t,
\end{align*}
where we used that $1-e^{-u}\leq u$. Thus, $\mathbf{1}_{x\sim 0}t-\bP_\xiz\big[\xi^0_t=\tau_x\xiz,n_*^\xiz(t)=1\big]$ is nonnegative. 

It follows from the Cauchy-Schwarz inequality that:
\begin{align}\label{EqVelo4}
\E_0\big[\big\Vert\bE_\xiz\big[X_t\mathbf{1}_{n_*^\xiz=1}\big]-t\varphi (\xiz)\big\Vert^2\big]
&=\E_0\Big[\Big\Vert\sum_{x\in\xiz}x\big(\bP_\xiz\big[\xi^0_t=\tau_x\xiz,n_*^\xiz(t)=1\big]-\mathbf{1}_{x\sim 0}t\big)\Big\Vert^2\Big]\nonumber\\
&\leq \E_0\big[S_1(\xiz)S_2(\xiz)\big],
\end{align}
with
\[S_1(\xiz):=\sum_{x\in\xiz}\big(\mathbf{1}_{x\sim 0}t-\bP_\xiz\big[\xi^0_t=\tau_x\xiz,n_*^\xiz(t)=1\big]\big),\]
and
\[S_2(\xiz):=\sum_{x\in\xiz}\Vert x\Vert^2\big(\mathbf{1}_{x\sim 0}t-\bP_\xiz\big[\xi^0_t=\tau_x\xiz,n_*^\xiz(t)=1\big]\big).\]

Let us observe that 
\begin{align}\label{EqVelo5}
S_1(\xiz)&=\deg_{G(\xiz)}(0)t-\bP_\xiz\big[n_*^\xiz(t)=1\big]\nonumber\\
&=\deg_{G(\xiz)}(0)t-\bP_\xiz\big[n_*^\xiz(t)\geq 1\big]+\bP_\xiz\big[n_*^\xiz(t)\geq 2\big]\nonumber\\
&=\deg_{G(\xiz)}(0)t-1+e^{-\deg_{G(\xiz)}(0)t}+\bP_\xiz\big[n_*^\xiz(t)\geq 2\big]\nonumber\\
&\leq\deg_{G(\xiz)}(0)^2t^2+\bP_\xiz\big[n_*^\xiz(t)\geq 2\big],
\end{align}
where we used that for $u\geq 0$, $0\leq -1+e^{-u}+u\leq u^2$. Moreover:
\begin{align}\label{EqVelo6}
\bP_\xiz\big[n_*^\xiz(t)\geq 2\big]&\leq \bP_0^\xiz\otimes\bQ\big[T_{0,0}^\xiz\leq t,\,T_{1,\X_1^\xiz}^\xiz\leq t\big]\nonumber\\
&=\big(1-e^{-\deg_{G(\xiz)}(0)t}\big)\sum_{x\in\xiz}\big(1-e^{-\deg_{G(\tau_x\xiz)}(0)t}\big)\bPt_\xiz\big[\xit^0_1=\tau_x\xiz\big]\nonumber\\
&\leq t^2\deg_{G(\xiz)}(0)\bEt_\xiz\big[\deg_{G(\xit_1)}(0)\big],
\end{align}
thus, with (\ref{EqVelo5}) and (\ref{EqVelo6}):
\begin{equation}\label{EqVelo7}
S_1(\xiz)\leq t^2\Big(\deg_{G(\xiz)}(0)^2+\deg_{G(\xiz)}(0)\bEt_\xiz\big[\deg_{G(\xit_1)}(0)\big]\Big). 
\end{equation}

On the other hand, one has:
\begin{align}\label{EqVelo8}
S_2(\xiz )&=\sum_{x\in\xiz}\Vert x\Vert^2\big(\mathbf{1}_{x\sim 0}t-\bP_\xiz\big[\xi^0_t=\tau_x\xiz,n_*^\xiz(t)=1\big]\big)\leq t\sum_{x\sim 0}\Vert x\Vert^2.
\end{align}

Note that:
\begin{equation}\label{EqVeloA}
\E_0\Big[\deg_{G(\xiz)}(0)^2\sum_{x\sim 0}\Vert x\Vert^2\Big]
\leq \E_0\Big[\deg_{G(\xiz)}(0)^6\Big]^\frac{1}{2}
\E_0\Big[\big(\max_{x\sim 0}\Vert x\Vert\big)^4\Big]^\frac{1}{2}.
\end{equation}

Thanks to the definition of $\bPt$ in Subsection \ref{ChaineIncluse}, we can write
\begin{align*}
\E_0\Big[\deg_{G(\xiz)}(0)\bEt_\xiz&\big[\deg_{G(\xit_1)}(0)\big]\sum_{x\sim 0}\Vert x\Vert^2\Big]\nonumber\\
&= \E_0\big[\deg_{G(\xiz)}(0)\big]\bEt\Big[\deg_{G(\xit_1)}(0)\sum_{x\sim 0}\Vert x\Vert^2\Big]\nonumber\\
&\leq\E_0\big[\deg_{G(\xiz)}(0)\big]\bEt\Big[\deg_{G(\xit_1)}(0)^2\Big]^\frac{1}{2}\bEt\Big[\Big(\sum_{x\sim 0}\Vert x\Vert^2\Big)^2\Big]^\frac{1}{2}.
\end{align*}
Due to the stationarity of the discrete-time environment process under $\bPt$, one has:
\begin{align}\label{EqVeloB}
\E_0\Big[\deg_{G(\xiz)}&(0)\bEt_\xiz\big[\deg_{G(\xit_1)}(0)\big]\sum_{x\sim 0}\Vert x\Vert^2\Big]\nonumber\\
\leq&\E_0\big[\deg_{G(\xiz)}(0)\big]^\frac{1}{2}\bEt\Big[\deg_{G(\xi^0_0)}(0)^2\Big]^\frac{1}{2}\E_0\Big[\Big(\sum_{x\sim 0}\Vert x\Vert^2\Big)^2\deg_{G(\xi^0)}(0)\Big]^\frac{1}{2}\nonumber\\
\leq&\E_0\Big[\deg_{G(\xiz)}(0)^3\Big]^\frac{1}{2}\E_0\Big[\deg_{G(\xiz)}(0)^3\big(\max_{x\sim 0}\Vert x\Vert\big)^4\Big]^\frac{1}{2}\nonumber\\
\leq &\E_0\Big[\deg_{G(\xiz)}(0)^3\Big]^\frac{1}{2}\E_0\Big[\deg_{G(\xiz)}(0)^6\Big]^\frac{1}{4}\E_0\Big[\big(\max_{x\sim 0}\Vert x\Vert\big)^8\Big]^\frac{1}{4}.
\end{align}  
Thanks to (\ref{EqVelo4}), (\ref{EqVelo7}), (\ref{EqVelo8}), (\ref{EqVeloA}) and (\ref{EqVeloB}), the first summand of (\ref{EqVelo2}) goes to 0 as $t\rightarrow 0$.

Let us move to the second term in the r.h.s.\@\xspace of (\ref{EqVelo2}). Applying the H\"older inequality with exponents $4$ and $4/3$ and the Cauchy-Schwarz inequality, we obtain:
\begin{align}\label{EqVelo9}
\E_0\Big[\big\Vert \bE_\xiz\big[X_t\mathbf{1}_{n^{\xi^0}_*(t)\geq 2}\big]\big\Vert^2\Big]
&\leq \E_0\Big[ \bE_\xiz\big[\big\Vert X_t\Vert^4\big]^\frac{1}{2}\bP_\xiz\big[n^{\xi^0}_*(t)\geq 2\big]^\frac{3}{2}\Big]\nonumber\\
&\leq \E_0\Big[ \bE_\xiz\big[\big\Vert X_t\Vert^4\big]\Big]^\frac{1}{2}\E_0\Big[\bP_\xiz\big[n^{\xi^0}_*(t)\geq 2\big]^3\Big]^{\frac{1}{2}}.
\end{align}
One can obtain as in (\ref{EqVelo6}) the alternative upper bound:
\begin{align}\label{EqVelo10}
\bP_\xiz\big[n_*^\xiz(t)\geq 2\big]
&\leq t\bEt_\xiz\big[\deg_{G(\xit^0_1)}(0)\big].
\end{align}
Hence, $(\ref{EqVelo6})$ multiplied $(\ref{EqVelo10})$ to the square, the Jensen's inequality and the stationarity of the discrete-time environment process under $\bPt=\int \bPt_\xiz \Q_0 (\d \xiz)$ imply that:
\begin{align}\label{EqVelo11}
\E_0\Big[\bP_\xiz\big[n^{\xi^0}_*(t)\geq 2\big]^3\Big]
&\leq t^4\E_0\Big[\deg_{G (\xiz)}(0)\bEt_\xiz\big[\deg_{G (\xit_1^0)}(0)\big]^3\Big]\nonumber\\
&\leq t^4\E_0\Big[\deg_{G (\xiz)}(0)\bEt_\xiz\Big[\big(\deg_{G (\xit_1^0)}(0)\big)^3\Big]\Big]\nonumber\\
&= t^4\E_0\big[\deg_{G (\xiz)}(0)\big]\bEt\Big[\big(\deg_{G (\xit_1^0)}(0)\big)^3\Big]\nonumber\\
&= t^4\E_0\Big[\big(\deg_{G (\xi^0)}(0)\big)^4\Big].
\end{align}  

Finally, with (\ref{EqRemIntegrability}), (\ref{EqVelo9}) and (\ref{EqVelo11}), we obtain that for $t$ small:
\begin{align*}
\frac{1}{t^2}\E_0\Big[\big\Vert \bE_\xiz\big[X_t\mathbf{1}_{n^{\xi^0}_*(t)\geq 2}\big]\big\Vert^2\Big]&\leq \const{9} t^\frac{1}{4},
\end{align*}
which tends to 0 with $t$.
\end{dem}
\subsection{Square integrability of the martingale $M_t=X_t-\int_0^t\varphi(\xi^0_s)\d s$}\label{MartL2} 
Due to Lemma \ref{lemmMoments}, $X_t\in L^2(\Xi^0,\bP)$, so it is enough to see that:
\[\int_0^t\varphi(\xi^0_s)\d s\in L^2(\Xi^0,\bP).\]
Using the Cauchy-Schwarz inequality and the stationarity of the environment process under $\bP$, one has:
\begin{align*}
\bE\Big[\Big\Vert\int_0^t\varphi (\xi^0_s)\d s\Big\Vert^2\Big]&\leq t \bE\Big[\int_0^t\big\Vert\varphi (\xi^0_s)\big\Vert^2\d s\Big]=t\int_0^t\bE\big[\big\Vert\varphi (\xi^0_s)\big\Vert^2\big]\d s=t^2\bE\big[\big\Vert\varphi (\xi^0_0)\big\Vert^2\big]\\
&=t^2\E_0\big[\big\Vert\varphi\big\Vert^2\big]\leq t^2\E_0\big[\deg_{G(\xiz)}(0)^4\big]^\frac{1}{2}\E_0\big[\big(\max_{x\sim 0}\Vert x\Vert\big)^4\big]^\frac{1}{2}<\infty.
\end{align*}

\subsection{Infinitesimal square displacement as a $\mathbf{L^2}$-limit }
By following precisely the same proof as in Lemma \ref{lemmvelocity} with higher order moments estimates, one obtains: 

\begin{lemm}\label{lemmSquareDisplacement}
For $\xiz\in\N_0$, let $\psi$ be the $d\times d$ symmetric matrix  characterized by:
\begin{equation}\label{defpsi}
\big(a\cdot \psi(\xiz)a\big):=\sum_{x\sim 0} (a\cdot x)^2,\quad a\in \RR^d.
\end{equation}
Then, $\big(a\cdot \psi(\cdot)a\big)\in L^2(\N_0, \P_0)$ and:
\begin{equation*}
\frac{1}{t}\bE_{\cdot}\big[(a\cdot X_t)^2\big]\xrightarrow[t\rightarrow 0]{\, L^2(\N_0,\P_0)\,}
  \big(a\cdot \psi(\cdot )a\big).
\end{equation*}
\end{lemm}

Since its assumptions were verified in this section, Theorem 2.2 of \cite{DFGW} implies Proposition \ref{PropCV}. It gives moreover the following characterization of the diffusion matrix $D$ of the limiting Brownian motion: 
\begin{equation}\label{EqVaria1}
(a\cdot Da)=\E_0\big[(a\cdot \psi a)\big]-2\int_0^\infty \langle \varphi\cdot a, e^{t\L}\varphi\cdot a\rangle_{\P_0}, \quad a\in\RR^d, 
\end{equation}
where $\L$ is the generator of the environment process and $\varphi$, $\psi$ are defined in (\ref{defphi}) and (\ref{defpsi}) respectively.

\section[The diffusion coefficient is nondegenerate]{Proof of Proposition \ref{PropNonDege}: the diffusion coefficient is nondegenerate}\label{Nondege}  
From now on, we assume that the hypotheses of Proposition \ref{PropNonDege} are satisfied and we prove it implies the non-degeneracy of the limiting Brownian motion.
\subsection{Variational formula}\label{FormuleVariationnelle}
In order to show that the diffusion matrix of the limiting Brownian motion in Proposition \ref{PropCV} is nondegenerate, it is useful to obtain a variational characterization of this matrix. The use of variational formulas to approximate diffusion coefficients is rather classical. It appears in particular in papers based on the works of Kozlov \cite{Kozlov} or De Masi, Ferrari, Goldstein and Wick \cite{DFGW}. In principle, it allows to derive the nondegeneracy of the diffusion coefficient and to obtain numerical estimates for it. See \cite{FSS,BiskupLN} and references therein.  
\begin{lemm}\label{LemmeVaria}
The diffusion matrix of the limiting Brownian motion in Proposition \ref{PropCV} is given by:
\begin{equation}\label{EqVaria2}
(a\cdot Da)=\inf_{f\in L^\infty(\N_0,\P_0)}\int_{\N_0}\sum_{x\in\xiz} c_{0,x}^\xiz\big( (a\cdot x) +\nabla_xf (\xiz)\big)^2\P_0 (\d \xiz ), \quad a\in\RR^d, 
\end{equation}
where $\nabla_x$ is the discrete gradient $\nabla_xf (\xiz):=f(\tau_x\xiz)-f(\xiz)$.
\end{lemm}
\begin{dem}
Minor changes to the proof of \cite[Proposition 3]{FSS} show that the generator $\L$ of the environment process is non-positive and self-adjoint with core $L^\infty (\N_0,\P_0)$. Moreover, for any $f\in L^\infty (\N_0,\P_0)$, one has:
\[\L f(\xiz)=\sum_{x\in\xiz}c_{0,x}^\xiz\big(f(\tau_x\xiz)-f(\xiz)\big),\]
where $c_{0,x}^\xiz$ equals the indicator function of the event `$\{0,x\}\in E_{G(\xiz)}$'. We refer for example to \cite[Chapter X]{ReedSimonII} for basic tools to prove this.
 
Hence, following the same strategy as in \cite{FSS} and invoking equations (45) and (46) of this paper, one obtains:
\begin{equation}\label{EqVaria3}
\int_0^\infty\langle \varphi\cdot a,e^{t\L}\varphi\cdot a\rangle_{\P_0}\d t=\sup_{f\in L^\infty(\N_0,\P_0)}\Big(2\langle\varphi\cdot a,f\rangle_\Pz-\langle f,(-\L)f\rangle_\Pz\Big).
\end{equation}

On the other hand, Neveu's exchange formula gives:
\begin{align*}
\int_\Nz \sum_{x\in\xiz}c_{0,x}^\xiz f(\xiz)^2 \P_0 (\d\xiz)&=\int_\Nz \sum_{x\in\xiz}c_{0,-x}^{\tau_x\xiz} f(\tau_x\xiz)^2 \P_0 (\d\xiz)\\
&=\int_\Nz \sum_{x\in\xiz}c_{0,x}^\xiz f(\tau_x\xiz)^2 \P_0 (\d\xiz)
\end{align*}
thus,
\begin{align}\label{EqVaria4}
\langle f, (-\L)f\rangle_\Pz &=\int_\Nz f (\xiz)\sum_{x\in\xiz}c_{0,x}^\xiz\big(f(\xiz)-f(\tau_x\xiz)\big)\Pz (d\xiz)\nonumber\\
&=\frac{1}{2}\int_\Nz \sum_{x\in\xiz}c_{0,x}^\xiz\big(f(\tau_x\xiz)^2+f(\xiz)^2-2f(\tau_x\xiz)f(\xiz)\big)\Pz (d\xiz)\nonumber\\
&=\frac{1}{2}\int_\Nz \sum_{x\in\xiz}c_{0,x}^\xiz\big(\nabla_x f(\xiz )\big)^2\Pz (d\xiz).
\end{align}

In the same way, Neveu's exchange formula gives:
\[\int_\Nz \sum_{x\in\xiz}c_{0,x}^\xiz (a\cdot x) f(\xiz) \P_0 (\d\xiz)
=-\int_\Nz \sum_{x\in\xiz}c_{0,x}^\xiz (a\cdot x)f(\tau_x\xiz) \P_0 (\d\xiz)
\]
which implies that:
\begin{align}\label{EqVaria5}
\langle \varphi\cdot a, f\rangle_\Pz &=\int_\Nz \sum_{x\in\xiz}c_{0,x}^\xiz(a\cdot x) f(\xiz)\Pz (\d\xiz)\nonumber\\
&=-\frac{1}{2}\int_\Nz \sum_{x\in\xiz}c_{0,x}^\xiz(a\cdot x)\nabla_x f(\xiz)\Pz (\d\xiz).
\end{align}

Finally, with equations (\ref{EqVaria1}), (\ref{EqVaria3})-(\ref{EqVaria5}), and the definition of $\psi$, one has:
\begin{align*}
(a\cdot Da)&=\E_0\big[(a\cdot \psi a)\big]-2\sup_{f\in L^\infty(\N_0,\P_0)}\Big(2\langle\varphi\cdot a,f\rangle_\Pz-\langle f,(-\L)f\rangle_\Pz\Big)\\
&=\inf_{f\in L^\infty(\N_0,\P_0)}\Big(\E_0\big[(a\cdot \psi a)\big]-4\langle\varphi\cdot a,f\rangle_\Pz+2\langle f,(-\L)f\rangle_\Pz\Big)\\
&=\inf_{f\in L^\infty(\N_0,\P_0)}\int_\Nz\sum_{x\in\xiz}c_{0,x}^\xiz\Big\{(a\cdot x)^2+2(a\cdot x)\nabla_xf(\xiz )+\big(\nabla_xf(\xiz )\big)^2\Big\}\Pz (\d\xiz )\\
\end{align*}
which gives the required result.
\end{dem}

For sake of simplicity, we assume from now on that the underlying point process is isotropic. Nevertheless, the arguments below can be easily adapted to the anisotropic case. 

\subsection{Cut-off on the transition rates and lower bound for the diffusion matrix}\label{CutOff}

In the sequel, we will consider the (largest) subgraph of $G(\xiz)$ in which each edge has length at most $r_c>0$. For $r_c>0$ fixed, we set:
\begin{equation}\label{EqCutOff1}
\widehat{c}_{x,y}^\xiz :=\mathbf{1}_{\{x,y\}{\tiny\mbox{ is an edge of }}G(\xiz){\tiny\mbox{ and }}\Vert y-x\Vert \leq r_c}\leq c_{x,y}^\xiz.
\end{equation}

Thanks to the variational formula (\ref{EqVaria2}), we immediately obtain that: 
\begin{align}\label{EqCutOff2}
(a\cdot Da)
&\geq\inf_{f\in L^\infty(\N_0,\P_0)}\int_{\N_0}\sum_{x\in\xiz} \widehat{c}_{0,x}^\xiz\big( (a\cdot x) +\nabla_xf (\xiz)\big)^2\P_0 (\d \xiz ).
\end{align}

As above, one can show that the operator $\widehat{\L}$  defined by:
\[\widehat{\L}f(\xiz):=\sum_{x\in\xiz}\widehat{c}_{0,x}^\xiz \big(f(\tau_x\xiz)-f(\xiz)\big),\quad f\in\mathcal{D}(\widehat{\L}),\]
has core $L^\infty (\N_0,\P_0)$, is self-adjoint and is the generator of a Markov process $(\hat{\xi}^0_t)_{t\geq 0}$ on $D(\RR_+;\Nz)$. We write $\widehat{\mathbf{P}}_\xiz$ (resp. $\widehat{\mathbf{P}}$) for the probability measure on $D(\RR_+;\Nz)$ associated with this Markov process with initial distribution $\delta_\xiz$ (resp. $\P_0$). The same arguments as in Subsection \ref{FormuleVariationnelle} show that the r.h.s.\@\xspace of (\ref{EqCutOff2}) is actually equal to:
\begin{equation}\label{EqCutOff3}
( a\cdot \widehat{D}a):=\E_0\big[(a\cdot\widehat{\psi} a)\big]-2\int_0^\infty\langle \widehat{\varphi}\cdot a,e^{t\widehat{\L}}\widehat{\varphi}\cdot a\rangle_\Pz\d t, 
\end{equation}
where
\[\widehat{\varphi}(\xiz):=\sum_{x\in\xiz}\widehat{c}_{0,x}^\xiz x\mbox{ and } (a\cdot \widehat{\psi} (\xiz )a):=\sum_{x\in\xiz}\widehat{c}_{0,x}^\xiz (a\cdot x)^2.\]
Hence, (\ref{EqCutOff2}) is equivalent to the inequality $D\geq\widehat{D}$.
\subsection{Periodic approximants and electrical networks} The aim of this subsection is to bound from below $\widehat{D}$ by using diffusion coefficients of random walks in appropriate periodic electrical networks. The method is very similar to the one in \cite{DFGW, FSS} and relies on the  theory of electrical networks (see \cite{SW,Lyons}) and homogenization (see \cite{JKO,Owhadi}). Actually, we will define suitable random resistor networks with diffusion coefficient which can be expressed as an effective conductance and which is lower than $\widehat{D}$. In the sequel, the medium will be periodized along the first coordinate direction.

\subsubsection{Periodized medium}\label{PeriodizedMedium}
For $\xi\in\N$ a locally finite subset of $\RR^d$ and $N>r_c$, we will construct two finite networks from $G(\xi)$. To this end, let us define the following subsets of $\RR^d$:
\[C_{2N}:=[-N,N]^d,\quad \oC_{2N}:=]-N,N[^d, \quad \gQ_N^\xi:=\xi \cap\oC_{2N},\]
\[\Gamma_N^\pm:=\big\{x=(x_1,\dots,x_d)\in\ZZ^d:\,x_1=\pm N,\,\vert x_j\vert <N, \,j=2,\dots,d\big\},\]
\[\gVb_N^\xi:=\gQ_N^\xi\cup \Gamma_N^-\cup\Gamma_N^+,\quad B_N^-:=\big\{x\in\oC_{2N}:\,x_1\in ]-N,-N+r_c] \big\},\]
\[B_N^+:=\big\{x\in\oC_{2N}:\,x_1\in [N-r_c,N[ \big\},\quad B_N^{\xi,\pm}:=\gQ_N^\xi\cap B_N^\pm.\]
A finite (unoriented) graph $\gGb_N^\xi=(\gVb_N^\xi,\gEb_N^\xi)$ is then constructed; it has an edge between $x$ and $y$ if:
\begin{itemize}
\item $x,y\in\gQ_N^\xi$, $\{x,y\}$ is an edge of $G(\xi)$ and $\Vert x-y\Vert \leq r_c$,
\item $x\in B_N^{\xi,-}$ and $y\in\Gamma_N^-$, 
\item $x\in B_N^{\xi,+}$ and $y\in\Gamma_N^+$.
\end{itemize}

\begin{figure}\label{FigconstGN}
\includegraphics[scale=0.5]{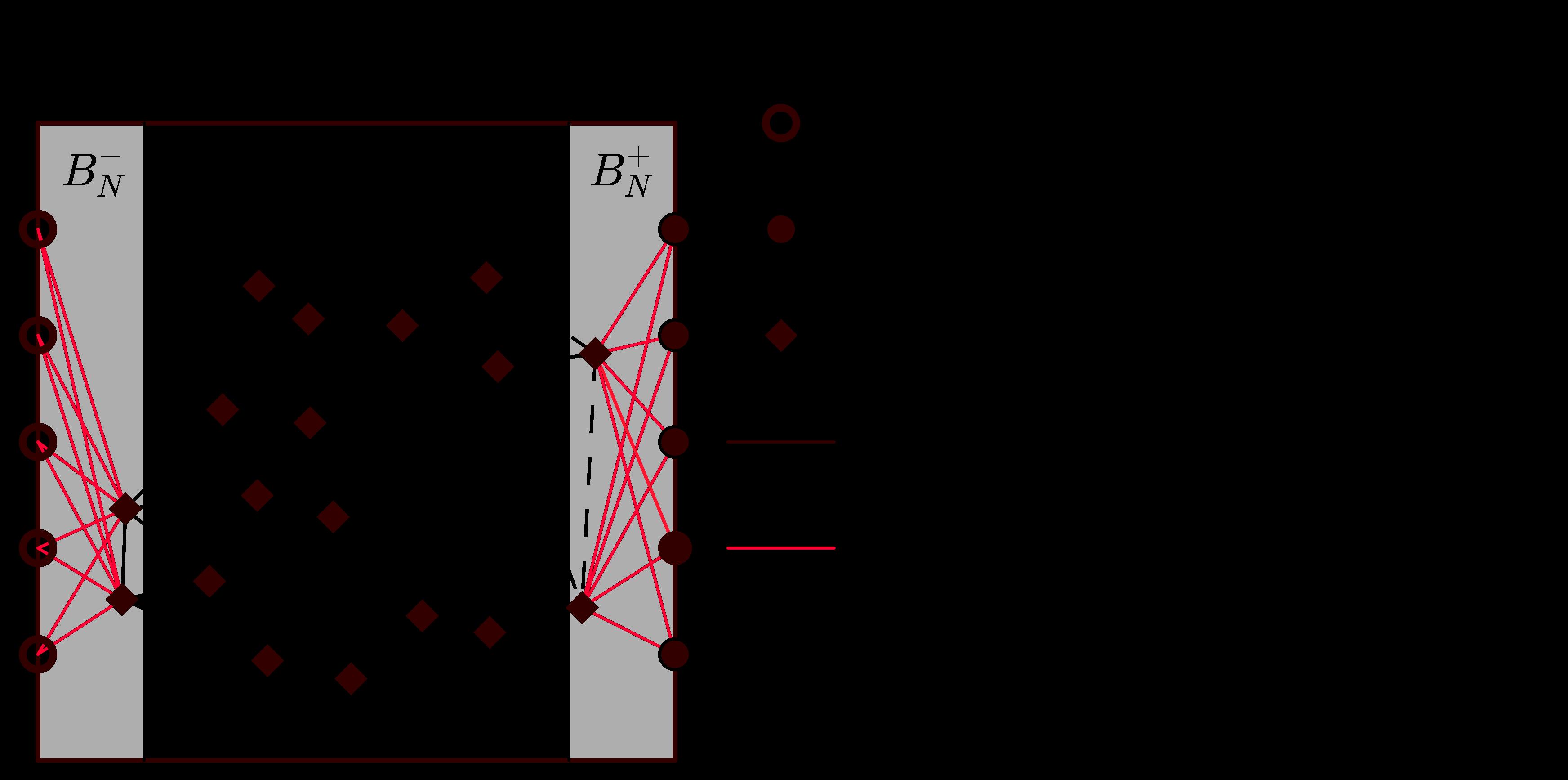}
\caption{The vertex set of $\gGb_N^\xi$ contains the points of $\gQ_N^\xi=\xi \cap\oC_{2N}$, $\Gamma_N^-$ and $\Gamma_N^+$. An edge of $G(\xi)$ between points of $\mathfrak{Q}^\xi_N$ is an edge of $\gGb_N^\xi$ {\it iff} its Euclidean length is smaller than $r_c$. There is also an edge in $\gGb_N^\xi$ between each point of $\xi\cap B_N^{-}$ (resp. $\xi\cap B_N^{+}$) and each point of $\Gamma_N^-$ (resp. $\Gamma_N^+$).}
\end{figure}

Another graph $\gG_N^\xi=(\gV_N^\xi,\gE_N^\xi)$ is obtained from $\gGb_N^\xi$ by identifying vertices:
\[x_-=(-N, x_2,\dots,x_d)\mbox{ and }x_+=(N, x_2,\dots,x_d).\] 
We denote by $\gp$ the projection map from $\gVb_N^\xi$ onto $\gV_N^\xi$. Note that $\gp_{\vert \gQ_N^\xi}$ is nothing but the identity. The graph $\gG_N^\xi$ will represent the \emph{periodized medium} in which any vertex in $\Gamma_N:=\gp (\Gamma_N^-)(=\gp (\Gamma_N^+))$ is connected to any vertex in $B_N^\xi:=B_N^{\xi,-}\cup B_N^{\xi,+}$. 

We consider now a continuous-time random walk on $\gV_N^\xi$ with infinitesimal generator given by:
\begin{equation*}
\L_N^\xi f(x):=\sum_{y\in \gV_N^\xi}\gc_N^\xi(x,y)\big(f(y)-f(x)\big),\quad \forall x\in\gV_N^\xi,
\end{equation*}
where
\begin{equation}\label{EqCondPeriodiques}
\gc_N^\xi(x,y):=\left\lbrace\
\begin{array}{ll}
1 &\mbox{if }\{ x,y\}\in\gE_N^\xi\mbox{ with }x,y\in\gQ_N^\xi\\
\dfrac{1}{\#(\Gamma_N)} &\mbox{if }\{ x,y\}\in\gE_N^\xi\mbox{ with }x\mbox{ or }y\in\Gamma_N\\
0 &\mbox{if }\{ x,y\}\not\in\gE_N^\xi
\end{array}\right. 
.
\end{equation}

Since $\mathfrak{c}^\xi_N(x,y)=\mathfrak{c}^\xi_N(y,x)$, the generator $\L_N^\xi$ is self-adjoint in $L^2(\gV_N^\xi,\gm_N^\xi)$ where $\gm_N^\xi$ is the uniform distribution on $\gV_N^\xi$. In general $\gG_N^\xi$ is not connected and $\gm_N^\xi$ is not ergodic. In this case, ergodic probability measures are uniform distributions on the connected components of $\gG_N^\xi$ and one can apply the results of De Masi, Ferrari, Goldstein and Wick in each connected component. We denote by $\bP_N^\xi$ (resp. $\bP_{N,x}^\xi,\, x\in\gV_N^\xi$) the distribution on the path space $\Omega_N^\xi:=D(\RR_+;\gV_N^\xi)$ of the random walk with generator $\L_N^\xi$ and initial distribution $\gm_N^\xi$ (resp. $\delta_x$). In order to apply one more time \cite[Theorem 2.2]{DFGW}, let us introduce $d_1$ the antisymmetric function on $\gV_N^\xi$ such that:
\begin{equation*}
d_1(x,y):=\left\lbrace\
\begin{array}{ll}
y_1-x_1 &\mbox{if }x,y\in\gQ_N^\xi\\
y_1+N &\mbox{if }y\in\gQ_N^\xi,\, y_1<0\mbox{ and }x\in\Gamma_N\\
y_1-N &\mbox{if }y\in\gQ_N^\xi,\, y_1>0\mbox{ and }x\in\Gamma_N
\end{array}\right. .
\end{equation*}
Moreover,  for $t\geq 0$, $\underline{\omega}=(\omega_s)_{s\geq 0}\in\Omega_N^\xi$, we define the random variable:
\begin{equation*}
X_{N,t}^{(1),\xi}(\underline{\omega}):=\sum_{\tiny
\begin{array}{c}
s\in [0,t]\\
\omega_{s^-}\neq\omega_s
\end{array}
}d_1(\omega_{s^-},\omega_s).
\end{equation*} 
The last quantity is the sum of the increments of the projections along the first coordinate of the positions taken by the particle up to time $t$. It is then possible to apply \cite[Theorem 2.2]{DFGW} to the time-covariant antisymmetric family $\{X_{N,[s,t]}^{(1),\xi}:=X_{N,t}^{(1),\xi}-X_{N,s}^{(1),\xi}\}_{t>s\geq 0}$ and obtain in particular:
\begin{lemm}\label{LemmPeriodizedMedium}
For $N\in\NN$, $N>r_c$ and $\xi\in\N$:
\begin{equation*}
\lim_{t\rightarrow\infty}\frac{1}{t}\bE_N^\xi\big[(X_{N,t}^{(1),\xi})^2\big]=D_N^\xi:=\gm_N^\xi \big(\psi_N^\xi \big)-2\int_0^\infty \langle \varphi_N^\xi,e^{t\L_N^\xi}\varphi_N^\xi \rangle_{\gm_N^\xi} \d t,
\end{equation*}
where $\varphi_N^\xi$ and $\psi_N^\xi$ are the scalar functions defined by:
\begin{equation}\label{EqLemmPeriodizedMedium2}\varphi_N^\xi (x):=\sum_{y\in\gV_N^\xi}\gc_N^\xi(x,y)d_1(x,y)\mbox{ and }\psi_N^\xi (x):=\sum_{y\in\gV_N^\xi}\gc_N^\xi(x,y)d_1(x,y)^2.
\end{equation}
\end{lemm}
\subsubsection{Relationship between $\widehat{D}$ and the diffusion constants of periodized media}
We compare $\widehat{D}$ (and therefore $D$) with suitable averages of the diffusion coefficients of the periodized media.
\begin{lemm}\label{LemmeComparaisonDesCoefficientsDeDiffusion}
It holds that:
\begin{equation}\label{EqComparaison2}
\lim_{N\rightarrow \infty}\E\big[\gm_N^\xi(\psi_N^\xi)\big]=\E_0\big[\widehat{\psi}_{1,1}\big]
\end{equation}
and
\begin{equation}\label{EqComparaison1}
\lim_{N\rightarrow \infty}\E\big[\langle \varphi_N^\xi,e^{t\L_N^\xi}\varphi_N^\xi\rangle_{\gm_N^\xi}\big]=\langle \widehat{\varphi}_1,e^{t\widehat{\L}}\widehat{\varphi}_1\rangle_{\P_0}
\end{equation}
where $\widehat{\varphi}_1$ (resp. $\widehat{\psi}_{1,1}$) is the first coordinate element of $\widehat{\varphi}$ (resp. the first diagonal element of $\widehat{\psi}$) defined in Subsection \ref{CutOff}.

In particular,
\begin{equation}\label{EqComparaison3}
D\geq \widehat{D}\geq \Big(\limsup_{N\rightarrow \infty}\E\big[D_N^\xi\big]\Big)\mathbf{I}_d.
\end{equation}
\end{lemm}
\begin{dem}
We first derive (\ref{EqComparaison3}) from (\ref{EqComparaison2}) and (\ref{EqComparaison1}). By the isotropy of $\P_0$, $\widehat{D}$ is a multiple of the identity matrix. Hence, using (\ref{EqCutOff3}) and Fatou's lemma:
\begin{align*}
D\geq \widehat{D}&=\Big(\E_0\big[\widehat{\psi}_{1,1}\big]-2\int_0^\infty \langle \widehat{\varphi}_1,e^{t\widehat{\L}}\widehat{\varphi}_1\rangle_{\P_0}\d t\Big)\mathbf{I}_d\\
&=\Big(\lim_{N\rightarrow \infty}\E\big[\gm_N^\xi(\psi_N^\xi)\big]-2\int_0^\infty \lim_{N\rightarrow \infty}\E\big[\langle \varphi_N^\xi,e^{t\L_N^\xi}\varphi_N^\xi\rangle_{\gm_N^\xi}\big]\d t\Big)\mathbf{I}_d\\
&\geq\Big(\limsup_{N\rightarrow \infty}\E\big[\gm_N^\xi(\psi_N^\xi)\big]-2\liminf_{N\rightarrow \infty}\int_0^\infty \E\big[\langle \varphi_N^\xi,e^{t\L_N^\xi}\varphi_N^\xi\rangle_{\gm_N^\xi}\big]\d t\Big)\mathbf{I}_d\\
&= \Big(\limsup_{N\rightarrow \infty}\E\big[D_N^\xi\big]\Big)\mathbf{I}_d.
\end{align*}

We now prove (\ref{EqComparaison2}). Recall the definition of $\widehat{c}_{x,y}^\xiz$ from (\ref{EqCutOff1}) and note that $\widehat{c}_{x,y}^\xiz=0$ if $x\in \oC_{2N-2r_c}$ and $y\not\in \oC_{2N}$. Thus, for $x\in \gQ_{N-r_c}$:
\begin{align*}
\widehat{\psi}_{1,1}(\tau_x\xi)=\big(e_1\cdot \widehat{\psi}(\tau_x\xi)e_1)&=\sum_{y\in\tau_x\xi}\mathbf{1}_{\{0,y\}\mbox{ is an edge of } G(\tau_x\xi),\,\Vert y\Vert\leq r_c}y_1^2\\
&=\sum_{\tiny\begin{array}{c}
y\in\gV_N^\xi\\ \{x,y\}\in\gE_N^\xi
\end{array}
}\big(y_1-x_1\big)^2=\psi_N^\xi(x).
\end{align*}
Since $x\in \gQ_{N-r_c}$, its neighbors in $\gG_{N}$ are points of $\gQ_{N}$ and thus the last sum runs only on points of $\gQ_{N}$.
This allows us to write
\begin{equation*}
\gm_N^\xi \big(\psi_N^\xi\big)=\frac{1}{\#\big(\xi \cap C_{2N}\big)+\#\Gamma_N}\Bigg(\sum_{x\in\gQ_{N-r_c}^\xi}\widehat{\psi}_{1,1}(\tau_x\xi)+\sum_{x\in\gV_N^\xi\setminus\gQ_{N-r_c}^\xi}\psi_N^\xi(x)\Bigg)
\end{equation*}
and to conclude by taking the $\P$-expectations in this identity and by using Lemma \ref{LemmeUtilPoisson} (\ref{LemmeUtilPoisson3}) and Lemma \ref{LemmMeanBContribNegl} with $h_N^\xi=\psi_N^\xi$. Before moving to the proof of (\ref{EqComparaison1}), let us note that the same method shows that for $1\leq p\leq 4$:
\begin{equation}\label{EqComparaison9}
\lim_{N\rightarrow\infty}\E\big[\gm_N^\xi\big(\vert \varphi_N^\xi\vert^p\big)\big]=\E_0\big[\vert \widehat{\varphi}_1\vert^p\big]\leq r_c^p\E_0\big[\#\big(\xiz\cap B_2(0,r_c\big)^p\big]<\infty.
\end{equation}

In order to  prove (\ref{EqComparaison1}), let us fix $0<\alpha<1$, set $M:=N-\lfloor N^\alpha\rfloor$ and define the exit times
\begin{equation}\label{DefHittingTimes}
\mathcal{T}_N^\xi:=\inf \{t\geq 0:\, \omega_t\not\in C_{2N-2r_c}\},\quad \underline{\omega}=(\omega_t)_{t\geq 0}\in\Omega_N^\xi=D(\RR_+;\gV_N^\xi).
\end{equation}

Since $\big(e^{t\L_N^\xi}\varphi_N^\xi\big)(x)=\bE_{N,x}^\xi\big[\varphi_N^\xi(\omega_t)\big]$, we can write:
\begin{equation*}
\E\big[\langle \varphi_N^\xi,e^{t\L_N^\xi}\varphi_N^\xi\rangle_{\gm_N^\xi}\big]=\E\big[A_{1,N}^\xi+A_{2,N}^\xi+A_{3,N}^\xi\big],
\end{equation*}
where
\begin{align*}
A_{1,N}^\xi &:=\gm_N^\xi\Big(\mathbf{1}_{x\not\in C_{2M}}\varphi_N^\xi(x)\bE_{N,x}^\xi\big[\varphi_N^\xi(\omega_t)\big]\Big),\\
A_{2,N}^\xi &:=\gm_N^\xi\Big(\mathbf{1}_{x\in C_{2M}}\varphi_N^\xi(x)\bE_{N,x}^\xi\big[\mathbf{1}_{\mathcal{T}_N^\xi\leq t}\varphi_N^\xi(\omega_t)\big]\Big),\\
A_{3,N}^\xi &:=\gm_N^\xi\Big(\mathbf{1}_{x\in C_{2M}}\varphi_N^\xi(x)\bE_{N,x}^\xi\big[\mathbf{1}_{\mathcal{T}_N^\xi > t}\varphi_N^\xi(\omega_t)\big]\Big).
\end{align*}
Then (\ref{EqComparaison1}) will be established once we show:
\begin{equation}\label{EqComparaison11}
\lim_{N\rightarrow\infty}\E\big[A_{i,N}^\xi\big]=0,\,i=1,2\mbox{ and }\lim_{N\rightarrow\infty}\E\big[A_{3,N}^\xi\big]=\langle \widehat{\varphi}_1,e^{t\widehat{\L}}\widehat{\varphi}_1\rangle_{\P_0}.
\end{equation}

For the first limit in (\ref{EqComparaison11}), we use the Cauchy-Schwarz inequality and the stationarity of the process $\underline{\omega}$ under $\gm_N^\xi$ to write:
\begin{align*}
\hspace{-0,5cm}\big\vert\E\big[A_{1,N}^\xi\big]\big\vert&\leq \E\Big[\gm_N^\xi\big(\mathbf{1}_{x\not\in C_{2M}}\big\vert\varphi_N^\xi(x)\big\vert\bE_{N,x}^\xi\big[\big\vert\varphi_N^\xi(\omega_t)\big\vert\big]\big)\Big]\\
&\leq \E\Big[\gm_N^\xi\big(\gV_N^\xi\setminus\ C_{2M}\big)\Big]^\frac{1}{2}\E\Big[\gm_N^\xi\big(\varphi_N^\xi(x)^2\bE_{N,x}^\xi\big[\varphi_N^\xi(\omega_t)^2\big]\big)\Big]^\frac{1}{2}\\
&\leq \E\Bigg[\frac{\#\Gamma_N+\#\big(\xi\cap C_{2N}\setminus C_{2M}\big)}{\#\Gamma_N+\#\big(\xi\cap C_{2N}\big)}\Bigg]^\frac{1}{2}\E\Big[\gm_N^\xi\big(\varphi_N^\xi(x)^4\big)\Big]^\frac{1}{4}\E\Big[\gm_N^\xi\big(\bE_{N,x}^\xi\big[\varphi_N^\xi(\omega_t)^4\big]\big)\Big]^\frac{1}{4}\\
&=\E\Bigg[\frac{\#\Gamma_N+\#\big(\xi\cap C_{2N}\setminus C_{2M}\big)}{\#\Gamma_N+\#\big(\xi\cap C_{2N}\big)}\Bigg]^\frac{1}{2}\E\Big[\gm_N^\xi\big(\varphi_N^\xi(x)^4\big)\Big]^\frac{1}{2}.
\end{align*}
The first factor of the r.h.s.\@\xspace goes to 0 thanks to the dominated convergence theorem. Indeed, the quotien is bounded by 1 and goes a.s. to 0 since its numerator is of order $N^{d-1}\lfloor N^\alpha\rfloor$ while its denominator is of order $N^d$. The second factor remains bounded thanks to (\ref{EqComparaison9}). Let us turn our attention to the second limit in (\ref{EqComparaison11}). Applying twice the Cauchy-Schwarz inequality and using stationarity, one has:
\begin{align*}
\hspace{-0,5cm}\big\vert\E\big[A_{2,N}^\xi\big]\big\vert&\leq \E\Big[\gm_N^\xi\big(\mathbf{1}_{x\in C_{2M}}\big\vert\varphi_N^\xi(x)\big\vert\bE_{N,x}^\xi\big[\mathbf{1}_{\mathcal{T}_N^\xi\leq t}\big\vert\varphi_N^\xi(\omega_t)\big\vert\big]\big)\Big]\\
&\leq \E\big[\gm_N^\xi\big(\varphi_N^\xi(x)^2\big)\big]^\frac{1}{2}\E\Big[\gm_N^\xi\big(\mathbf{1}_{x\in C_{2M}}\bE_{N,x}^\xi\big[\mathbf{1}_{\mathcal{T}_N^\xi\leq t}\varphi_N^\xi(\omega_t)^2\big]\big)\Big]^\frac{1}{2}\\
&\leq \E\big[\gm_N^\xi\big(\varphi_N^\xi(x)^2\big)\big]^\frac{1}{2}\E\big[\gm_N^\xi\big(\bE_{N,x}^\xi\big[\varphi_N^\xi(x)^4\big]\big)\big]^\frac{1}{4}\E\Big[\gm_N^\xi\big(\mathbf{1}_{x\in C_{2M}}\bP_{N,x}^\xi\big[\mathcal{T}_N^\xi\leq t\big]\big)\Big]^\frac{1}{4}\\
&= \E\big[\gm_N^\xi\big(\varphi_N^\xi(x)^2\big)\big]^\frac{1}{2}\E\big[\gm_N^\xi\big(\varphi_N^\xi(x)^4\big)\big]^\frac{1}{4}\E\Big[\gm_N^\xi\big(\mathbf{1}_{x\in C_{2M}}\bP_{N,x}^\xi\big[\mathcal{T}_N^\xi\leq t\big]\big)\Big]^\frac{1}{4}.
\end{align*}
Once again, thanks to (\ref{EqComparaison9}), the first two terms are uniformly bounded. The last factor goes to 0 by Lemma \ref{LemmeHittingTimes}.

To deal with the last limit in (\ref{EqComparaison11}), we need to introduce the following exit times:  
\begin{equation}\label{DefHittingTimesEnvironment}
\mathcal{T}_{N,x}:=\inf \{t\geq 0:\, x+X_t(\underline{\hat{\xiz}})\not\in C_{2N-2r_c}\},\quad x\in \xiz\cap C_{2M},\,\underline{\hat{\xiz}}=(\hat{\xi}_t^0)_{t\geq 0}\in \Xi^0,
\end{equation}
with $X_t(\underline{\hat{\xiz}})$ as defined in (\ref{deffamcov}) and $\Xi^0$ as defined in Subsection \ref{PointDeVueDeLaparticle}.

Note that, if $N$ is large enough ({\it i.e.\@\xspace} $\lfloor N^\alpha\rfloor>r_c$), one has for $x\in C_{2M}$:
\[\varphi_N^\xi (x)=\widehat{\varphi}_1(\tau_x\xi)\mbox{ and }\bE_{N,x}^\xi\big[\mathbf{1}_{\mathcal{T}_N^\xi >t}\varphi_N^\xi (\omega_t)\big]=\widehat{\bE}_{\tau_x\xi}\big[\mathbf{1}_{\mathcal{T}_{N,x} >t}\widehat{\varphi}_1 (\hat{\xi}^0_t)\big].\]

Hence,
\begin{equation*}
\E\big[A_{3,N}^\xi\big]=\E\Big[\gm_N^\xi\Big(\mathbf{1}_{x\in C_{2M}}\widehat{\varphi}_1 (\tau_x\xi)\widehat{\bE}_{\tau_x\xi}\big[\mathbf{1}_{\mathcal{T}_{N,x} >t}\widehat{\varphi}_1 (\hat{\xi}^0_t)\big]\Big)\Big].
\end{equation*}

In the same spirit as what we did to obtain $\lim_{N\rightarrow\infty}\E[A_{2,N}^\xi]=0$, we get using Lemma \ref{LemmeHittingTimes} (\ref{EqLemmeHittinTimes2}):
\begin{equation*}
\lim_{N\rightarrow\infty}\E\Big[\gm_N^\xi\Big(\mathbf{1}_{x\in C_{2M}}\widehat{\varphi}_1 (\tau_x\xi)\widehat{\bE}_{\tau_x\xi}\big[\mathbf{1}_{\mathcal{T}_{N,x} \leq t}\widehat{\varphi}_1 (\hat{\xi}^0_t)\big]\Big)\Big]=0,
\end{equation*}
so that
\begin{equation*}
\lim_{N\rightarrow\infty}\E\big[A_{3,N}^\xi\big]=\lim_{N\rightarrow\infty}\E\Big[\gm_N^\xi\Big(\mathbf{1}_{x\in C_{2M}}\widehat{\varphi}_1 (\tau_x\xi)\widehat{\bE}_{\tau_x\xi}\big[\widehat{\varphi}_1 (\hat{\xi}^0_t)\big]\Big)\Big].
\end{equation*}

Finally, we apply Lemma \ref{LemmeUtilPoisson} (\ref{LemmeUtilPoisson4}) the function
\[h(\xiz):=\widehat{\varphi}_1 (\xiz)\widehat{\bE}_{\xiz}\big[\widehat{\varphi}_1 (\hat{\xi}^0_t)\big]=\widehat{\varphi}_1 (\xiz)\Big(e^{t\widehat{\L}}\widehat{\varphi}_1\Big) (\xiz).\] 
This function is in $L^2(\N_0,\P_0)$ since the process is stationary w.r.t.\@\xspace $\widehat{\bE}$ and  $\operatorname{deg}_{G(\xiz)}(0)$ admits polynomial moments under $\P_0$.
This gives
\begin{equation*}
\lim_{N\rightarrow\infty}\E\big[A_{3,N}^\xi\big]=\lim_{N\rightarrow\infty}\E\Big[\gm_N^\xi\Big(\mathbf{1}_{x\in C_{2M}}\widehat{\varphi}_1 (\tau_x\xi)\widehat{\bE}_{\tau_x\xi}\big[\widehat{\varphi}_1 (\hat{\xi}^0_t)\big]\Big)\Big]=\langle\widehat{\varphi}_1,e^{t\widehat{\L}}\widehat{\varphi}_1\rangle_{\P_0}.
\end{equation*}
\end{dem} 

\subsubsection{Relationship between $D_N^\xi$ and the effective conductance of suitable resistor networks}\label{DiffusionConductivity} We want to establish a link between the diffusion constant $D_N^\xi$ and the effective conductance of an appropriate electrical network. To this end, recall definitions of Subsection \ref{PeriodizedMedium} and equip edges of $\gGb_N^\xi$ with conductances:
\begin{equation*}
\gcb_N^\xi(x,y):=\gc_N^\xi(\gp (x),\gp (y)), \quad \{x,y\}\in\gEb_N^\xi,
\end{equation*}
where $\gc_N^\xi$ is given by (\ref{EqCondPeriodiques}) and $\gp$ is the projection map from $\gVb_N^\xi$ onto $\gV_N^\xi$. Then, the effective conductivity $\overline{\kappa}_N^\xi$ of this electrical network is defined as the total amount of current flowing from $\Gamma_N^-$ to $\Gamma_N^+$ when a unit potential difference is imposed between these two sets. From Dirichlet's principle (see \cite[\S 2.4]{Lyons} for example), we know that: 
\begin{equation}\label{EqDiffusioConductivity2}
\overline{\kappa}_N^\xi=\frac{1}{2}\inf \Bigg\{\sum_{x\in\gVb_N^\xi}\sum_{y:\,
\{x,y\}\in\gEb_N^\xi}
\gcb_N^\xi(x,y)\big(\nabla_{(x,y)}g\big)^2:\,g:\gVb_N^\xi\mapsto \RR,\,g_{\vert\Gamma_N^-}\equiv 0,\, g_{\vert\Gamma_N^+}\equiv 1 \Bigg\}
\end{equation}
where $\nabla_{(x,y)}g:=g(y)-g(x)$.

The same method as the one used in Subsection \ref{FormuleVariationnelle} shows that the diffusion coefficient $D_N^\xi$ (defined in Lemma \ref{LemmPeriodizedMedium}) satisfies the variational formula:
\begin{align}\label{EqDiffusionConductivity3}
D_N^\xi&=\frac{1}{\#\gV_N^\xi}\inf_{f\, :\, \gV_N^\xi\rightarrow \RR}\Bigg\{\sum_{x\in\gV_N^\xi}\sum_{y:\,\{x,y\}\in\gE_N^\xi}\gc_N^\xi(x,y)\big(d_1(x,y)+\nabla_{(x,y)}f\big)^2\Bigg\} \nonumber\\
&=\frac{1}{\#\gV_N^\xi}\inf_{\tiny
\begin{array}{c}
f:\gVb_N^\xi\rightarrow \RR\\
f\text{ constant on }\gp^{-1}(\{x\}),\forall x\in\Gamma_N
\end{array}}
\Bigg\{\sum_{x\in\gVb_N^\xi}\sum_{y:\,\{x,y\}\in\gEb_N^\xi}\gcb_N^\xi(x,y)\big(\nabla_{(x,y)}(\pi_1+f)\big)^2\Bigg\},
\end{align}  
where $\pi_1$ is the projection along the first coordinate. 

We claim that:
\begin{equation}\label{EqDiffusionConductivity4}
D_N^\xi=\frac{1}{\#\gV_N^\xi}\inf_{\tiny
\begin{array}{c}
f:\gVb_N^\xi\rightarrow \RR\\
f_{\vert\Gamma_N^-\cup\Gamma_N^+}=N
\end{array}}
\Bigg\{\sum_{x\in\gVb_N^\xi}\sum_{y:\,\{x,y\}\in\gEb_N^\xi}\gcb_N^\xi(x,y)\big(\nabla_{(x,y)}(\pi_1+f)\big)^2\Bigg\}.
\end{equation}

To see this, let us fix the values of such a function $f$ on $\mathfrak{Q}_N^\xi$, say $f_{\vert\mathfrak{Q}_N^\xi}=h_{\vert\mathfrak{Q}_N^\xi}$. We will verify that the function minimizing (\ref{EqDiffusionConductivity3}) subject to this condition is actually constant on $\Gamma_N^-\cup\Gamma_N^+$. Note that, when the values of $f$ on $\gQ_N^\xi$ are fixed,  we know the value of:
\[\sum_{x\in\gQ_N^\xi}\sum_{y\in\gQ_N^\xi}\mathbf{1}_{\{x,y\}\in\gEb_N^\xi}\big(\nabla_{(x,y)}(\pi_1+f)\big)^2,\]
and, writing $C_x$ for the value of $f$ on $\mathfrak{p}^{-1}(\{x\})$, $x\in \Gamma_N$, we are led to minimize in $C_x$:
\[\sum_{x\in\Gamma_N}\Bigg\{\sum_{y\in B_N^{\xi,-}}\big(y_1+N+h(y)-C_x\big)^2+\sum_{y\in B_N^{\xi,+}}\big(y_1-N+h(y)-C_x\big)^2\Bigg\}.\]
The summand depends on $x$ only through $C_x$ and we must choose $C_x$ as the (unique) minimizer of this quantity for each $x\in\Gamma_N$. Thus the second infimum in (\ref{EqDiffusionConductivity3}) is achieved for functions which are constant on $\Gamma_N^-\cup\Gamma_N^+$. Since $f$ appears in the variational formula in a gradient form, one can assume that $f$ is equal to $N$ on $\Gamma_N^-\cup\Gamma_N^+$ and (\ref{EqDiffusionConductivity4}) is proved.

Now, let us note that the functions $g:\gVb_N^\xi\mapsto \RR$ satisfying $g_{\vert\Gamma_N^-}\equiv 0$ and $g_{\vert\Gamma_N^+}\equiv 1$ can be expressed as:
\[g=\frac{1}{2N}\big(\pi_1+f\big),\]
where $f$ is a function on $\gVb_N^\xi$ such that $f_{\vert\Gamma_N^-\cup\Gamma_N^+}\equiv N$ and {\it vice versa}. Hence, 
\begin{equation*}
D_N^\xi=\frac{4N^2}{\# \big(\gV_N^\xi\big)}\inf \Bigg\{\sum_{x\in\gVb_N^\xi}\sum_{y:\,
\{x,y\}\in\gEb_N^\xi}
\gcb_N^\xi(x,y)\big(\nabla_{(x,y)}g\big)^2:\,g:\gVb_N^\xi\mapsto \RR,\,g_{\vert\Gamma_N^-}\equiv 0,\, g_{\vert\Gamma_N^+}\equiv 1 \Bigg\}.
\end{equation*}

Thanks to Dirichlet's principle (\ref{EqDiffusioConductivity2}), we have just proved:
\begin{lemm}\label{LemmeDiffusionConductivity}
\begin{equation*}\label{EqDiffusionConductivity5}
D_N^\xi=\frac{8N^2}{\# \big(\gV_N^\xi\big)}\overline{\kappa}_N^\xi.
\end{equation*}
\end{lemm}

Lower bounds on $\overline{\kappa}_N^\xi$ are obtained in Subsection \ref{PositiveConductivity}

\subsection{Lower bound on $\overline{\kappa}_N^\xi$}\label{PositiveConductivity}

In order to obtain lower bounds on $\overline{\kappa}_N^\xi$ for large $N$, we use the discretization of the space and the `\emph{good boxes}' introduced in Proposition \ref{PropNonDege}. These `good boxes' will allow us to link `nice properties' of the random graphs we consider with `nice properties' of independent Bernoulli processes on $\ZZ^d$. The rest of the proof then relies on techniques from percolation theory which also appear in \cite[Section 6]{FSS}. This method is actually quite classical and is partially described in the 2-dimensional case in \cite[Chapter 11]{KestenPerco}.
\subsubsection{Percolation estimates and LR-crossings}\label{PercolationEstimates}
Let us consider an independent Bernoulli site percolation process $\YY=\{Y_\z,\,\z\in\ZZ^d\}$ on $\ZZ^d$ with parameter $p>p^\text{site}_c(\ZZ^2)$. We will use fine estimates about the number of left-right crossings (LR-crossings) of $R_{2N,2(N-1)}:=[-N,N]\times [-N+1,N-1]^{d-1}$. Let us recall that a LR-crossing with length $k-1$ of $R_{2N,2(N-1)}$ for a realization of $\{Y_\z,\,\z\in\ZZ^d\}$ is a sequence of distinct points $(\z_1, \dots, \z_k)$ in $R_{2N,2(N-1)}\cap \ZZ^d$ such that $|\z_i-\z_{i+1}|=1$ for $1\leq i\leq k-1$, $Y_{\z_i}=1$ for $1\leq i \leq k$, $\z_1^{(1)}=-N$, $\z_k^{(1)}=N$, $-N<\z_i^{(1)}<N$ for $1<i<k$, and $\z_i^{(l)}=\z_j^{(l)}$ for any $l\geq 3$ and for $1\leq i<j\leq k$. Two crossings are said to be disjoint if all the involved $\z_j$s are  distinct. Note that, by definition, LR-crossings are contained into
$2$-dimensional slices. The techniques of \cite[Sections 2.6 and 11.3]{Grimmett} can be used, in the context of site
percolation, to derive:
\begin{prop}
If $p>p_c(2)$, there exist positive constants $a=a(p) $, $b=b(p)$, and $c=c(p)$ such that for all $N\in \NN^*$:
\begin{equation}\label{EqLRCrossing}  
\PP \big[ \YY \mbox{ has less than }b N^{d-1} \mbox{ disjoint LR-crossings in } R_{2N,2(N-1)}\big]\leq ce^{-a N}.
\end{equation}
\end{prop}

Thanks to the domination assumption (\ref{CondNonDege3}) in Proposition \ref{PropNonDege}, there exists a coupling of the process of the good boxes $\XX$ and the supercritical independent Benoulli procces $\YY$ on $\ZZ^d$ such that almost surely $Y_\z\leq X_\z$ for all $\z\in\ZZ^d$. In other words, any good box corresponds almost surely to an open site in the percolation process $\YY$.

From now on, we fix $r_c:=\sqrt{d+3}K$. Consider a  LR-crossing $\gamma^\mathrm{perco}=(\z_1,\dots ,\z_k)$ of $R_{2N,2(N-1)}$. One can associate to $\gamma^\mathrm{perco}$ a LR-crossing $\gamma^{\overline{\mathfrak{G}}_{KN}^\xi}$ of $\gGb_{KN}^\xi$ from a point of $\Gamma_{KN}^-$ to a point of $\Gamma_{KN}^+$. To construct this LR-crossing, we first link $K\z_1\in\Gamma^-_{KN}$ to the reference vertex $v_{\z_2}$ of $B_{\z_2}$. Since $v_{\z_2}\in B_{\z_2}$, $\Vert K{\z_1}-v_{\z_2}\Vert \leq r_c$ and $\{Kz_1,v_{\z_2}\}$ is an edge of $\gGb_{KN}^\xi$. Next, for $i=2,\dots,k-2$, we go from $v_{\z_i}$ to $v_{\z_{i+1}}$ using the path given in (\ref{CondNonDege2}) of the definition of the good boxes in Proposition \ref{PropNonDege}. These edges are edges of $\gGb_{KN}^\xi$ thanks to the choice of $r_c$.  Finally, we link $v_{\z_{k-1}}$ to $K\z_k\in\Gamma_{KN}^+$. We assume that $\gamma^{\overline{\mathfrak{G}}_{KN}^\xi}$ is self-avoiding; if not, we remove loops. Since the boxes $B_{\z_1},\dots, B_{\z_k}$ are good, it is easy to see that:
\begin{itemize}
\item[(a)] $\operatorname{L}(\gamma^{\overline{\mathfrak{G}}_{KN}^\xi})\leq L\times \operatorname{L}(\gamma^\mathrm{perco})$ where $\operatorname{L}(\cdot)$ denotes the graph length of a path,
\item[(b)] $\gamma^{\overline{\mathfrak{G}}_{KN}^\xi}\subset \bigcup_{\z\in\gamma^\mathrm{perco}}\big(B_\z\cap C_{2KN}\big)$,
\end{itemize}
in particular,
\begin{itemize}
\item[(c)] if $\gamma^\mathrm{perco}_1$ and $\gamma^\mathrm{perco}_2$ are two disjoint LR-crossings of $R_{2N,2(N-1)}$ for $\mathbb{X}$, then $\gamma^{\overline{\mathfrak{G}}_{KN}^\xi}_1$ and $\gamma^{\overline{\mathfrak{G}}_{KN}^\xi}_2$ are disjoint LR-crossings of $\gGb_{N}^\xi$. 
\end{itemize}
\subsubsection{Lower bound for $\overline{\kappa}_{KN}^\xi$ for $N$-good configurations}
We say that a configuration $\xi\in\N$ is \emph{$N$-good} if the corresponding realization of the percolation process $\YY$ has at least $bN^{d-1}$ disjoint LR-crossings of $R_{2N,2(N-1)}$. From now on, we assume that $\xi$ is an $N$-good configuration and we write $\mathcal{C}_N(\xi)$ for a collection of at least $bN^{d-1}$ such disjoint crossings. By the arguments given at the end of Subsection \ref{PercolationEstimates}, there are at least $bN^{d-1}$ LR-crossings of $C_{2KN}$ in the graph $\gGb_N^\xi$. We write as before $\gamma^{\overline{\mathfrak{G}}_{KN}^\xi}$ for the LR-crossing in $\gGb_N^\xi$ associated with the LR-crossing $\gamma^\mathrm{perco}$ of $R_{2N,2(N-1)}$ for $\YY$. Since the paths are assumed to be self-avoiding and the $\gamma^\mathrm{perco}$s in $\mathcal{C}_N(\xi)$ are disjoints, it holds that:
\[\sum_{\gamma^\mathrm{perco}\in\mathcal{C}_N(\xi)}\operatorname{L}(\gamma^{\overline{\mathfrak{G}}_{KN}^\xi})\leq L\sum_{\gamma^\mathrm{perco}\in\mathcal{C}_N(\xi)}L(\gamma^\mathrm{perco})\leq L(2N+1)^d.\]
 With Jensen's inequality, one can deduce:
\begin{equation}\label{EqLowerboundGood1}
\sum_{\gamma^\mathrm{perco}\in\mathcal{C}_N(\xi)}\frac{1}{\operatorname{L}(\gamma^{\overline{\mathfrak{G}}_{KN}^\xi})}\geq\frac{\#\big(\mathcal{C}_N(\xi)\big)^2}{\sum_{\gamma^\mathrm{perco}\in\mathcal{C}_N(\xi)}\operatorname{L}(\gamma^{\overline{\mathfrak{G}}_{KN}^\xi})}\geq\frac{\big(bN^{d-1}\big)^2}{L(2N+1)^d}\geq \const{10} (KN)^{d-2},
\end{equation}
with $\const{10}=\const{10}(d,K)$.

We are now ready to bound from below $\overline{\kappa}_{KN}^\xi$. To this end, let us introduce another electrical network $\check{\gG}_{KN}^\xi$ with vertex set $\check{\gV}_{KN}^\xi:=\gQ_{KN}^\xi\cup\{\check{\Gamma}_{KN}^-,\check{\Gamma}_{KN}^+\}$, where $\check{\Gamma}_{KN}^-,\check{\Gamma}_{KN}^+$ are two extra points. Its edge set $\check{\gE}_{KN}^\xi$ consists in the collection of edges of $\gEb_{KN}^\xi$ with both endpoints in $\gQ_{KN}^\xi$ and extra edges between $\check{\Gamma}_{KN}^\pm$ and each point of $B_{KN}^{\xi,\pm}$. We assign a conductance 1 to each of these edges and  we define the effective conductance $\check{\kappa}_{KN}^\xi$ of $\check{\gG}_{KN}^\xi$ as the total amount of current flowing from $\check{\Gamma}_{KN}^-$ to $\check{\Gamma}_{KN}^+$ when a unit difference of potential is imposed. This new network can be seen as obtained from $\gGb_{KN}^\xi$ by merging together vertices in $\Gamma_{KN}^-$ and $\Gamma_{KN}^+$ respectively. Since the same difference of potential was imposed between $\Gamma_{KN}^-$ and $\Gamma_{KN}^+$ in the definition of $\overline{\kappa}_{KN}^\xi$, it is not difficult to see that $\check{\kappa}_{KN}^\xi=\overline{\kappa}_{KN}^\xi$. 

The disjoint LR-crossing $\gamma^{\overline{\mathfrak{G}}_{KN}^\xi}$ of $\gGb_{KN}^\xi$ exhibited above can be used to bound $\check{\kappa}_{KN}^\xi$ from below. They correspond to disjoint (except at their endpoints) paths from $\Gamma_{KN}^-$ to $\Gamma_{KN}^+$ in the new network. By \emph{Rayleigh monotonicity principle}, one can obtain a lower bound on $\check{\kappa}_{KN}^\xi$ by removing each edge which does not appear in these crossings. Thanks to the \emph{series law} each of these crossings $\gamma^{\overline{\mathfrak{G}}_{KN}^\xi}$ has `total' resistance $\operatorname{L}(\gamma^{\overline{\mathfrak{G}}_{KN}^\xi})$. Finally, due to \emph{parallels law} and (\ref{EqLowerboundGood1}), one has:
\begin{equation}\label{EqLowerboundGood2}
\check{\kappa}_{KN}^\xi=\overline{\kappa}_{KN}^\xi\geq \sum_{\gamma^\mathrm{perco}\in\mathcal{C}_N(\xi)}\frac{1}{\operatorname{L}(\gamma^{\overline{\mathfrak{G}}_{KN}^\xi})}\geq \const{10} (KN)^{d-2}.
\end{equation}
\subsection{Conclusion by collecting bounds}\label{Conclusion}
Due to (\ref{EqComparaison3}), it remains to show that:
\begin{equation*}
\limsup_{N\rightarrow \infty}\E\big[D_N^\xi\big]>0.
\end{equation*}

But, with Lemma \ref{LemmeDiffusionConductivity} and (\ref{EqLowerboundGood2}), we obtain:
\begin{align*}
\hspace{-0,3cm}\limsup_{N\rightarrow \infty}\E\big[D_N^\xi\big]&\geq 8K^2\limsup_{N\rightarrow \infty}\E\Bigg[\frac{N^2}{\#\big(\gV_{KN}^\xi\big)}\overline{\kappa}_{KN}^\xi\Bigg]\nonumber\geq \const{11}\lim_{N\rightarrow \infty}\E\Bigg[\frac{(KN)^d}{\#\big(\gV_{KN}^\xi\big)}\mathbf{1}_{\xi\mbox{ is }N\mbox{-good}}\Bigg],
\end{align*}
where we used that:
\[\E\Bigg[\frac{(KN)^2}{\#\big(\gV_{KN}^\xi\big)}\overline{\kappa}_{KN}^\xi\mathbf{1}_{\xi\mbox{ is }N\mbox{-bad}}\Bigg]\geq 0.\]

Observe that thanks to (\ref{EqLRCrossing}):
\begin{align*}
\E\Bigg[\frac{(KN)^d}{\#\big(\gV_{KN}^\xi\big)}\mathbf{1}_{\xi\mbox{ is }N\mbox{-bad}}\Bigg]&=\E\Bigg[\frac{(KN)^d}{\#\big(\xi\cap C_{2KN}\big)+\#\big(\Gamma_{2KN}\big)}\mathbf{1}_{\xi\mbox{ is }N\mbox{-bad}}\Bigg]\nonumber\\
&\leq \const{12}N^d\P\big[\xi\mbox{ is }N\mbox{-bad}\big]\leq \const{13}N^de^{-aN}\xrightarrow[\,N\rightarrow\infty\,]{}0.
\end{align*}

Thus, we obtain that:
\begin{equation*}
\limsup_{N\rightarrow \infty}\E\big[D_N^\xi\big]\geq \const{11}\lim_{N\rightarrow \infty}\E\Bigg[\frac{(KN)^d}{\#\big(\gV_{KN}^\xi\big)}\Bigg]>0
\end{equation*}
thanks to Lemma \ref{LemmeUtilPoisson} (\ref{LemmeUtilPoisson1}). Hence, the diffusion coefficient of the limiting Brownian motion in Proposition \ref{PropCV} is nondegenerate and Proposition \ref{PropNonDege} is proved.

\section{Derivation of Theorem \ref{ThprincAIP} from Propositions \ref{PropCV} and \ref{PropNonDege}}\label{DerivTh1}
In this section, we show that the hypotheses listed in Subsection \ref{condprocAIP} imply the assumptions of Propositions \ref{PropCV} and \ref{PropNonDege} in the cases of Delaunay triangulations, Gabriel graphs and \emph{creek-crossing} graphs. Theorem \ref{ThprincAIP} then immediately follows from these propositions.
\subsection{Polynomial moments for $\deg_{\DT(\xi^0)}(0)$ and $\max_{x\sim 0 \text{ in }\DT(\xi^0)}(0)$ under $\P_0$}\label{AppMoments}
In order to apply Proposition \ref{PropCV} to random walks on Delaunay triangulations, Gabriel graphs and \emph{creek-crossing} graphs, we only need to show that the assumptions of Subsection \ref{condprocAIP} imply that $\deg_{G(\xi^0)}(0)$ and $\max_{x\sim 0}(0)$ admit polynomial moment of any order under $\P_0$. Since the Gabriel graph and the \emph{creek-crossing} graphs are sub-graphs of the Delaunay triangulation, it is enough to do so in the case of Delaunay triangulations.

\begin{lemm}\label{LemmMomentsDegree}
If the point process satisfies {\bf (V')} and {\bf (D')}, then $\deg_{\DT(\xi^0)}(0)$ and $\max_{x\sim 0 \text{ in }\DT(\xi^0)}(0)$ admit polynomial moments of any order under $\P_0$.
\end{lemm}
\begin{proof} Let us consider a collection $\{C^L_j,j\in J_L\}$ of cubes of side $L$ that, as in Figure \ref{DessinDegre}, surround $\mathbf{C}_L=[-(2\lceil\sqrt{d}\rceil+1) L,(2\lceil\sqrt{d}\rceil+1) L]^d$. Note that $\#J_L=O(L^{d-1})$. Let $\mathcal{A}_L$ be the event `{\it $\#(C^L_j\cap\xi^0)\geq 1$ for all $j\in J_L$ and $\#(\overline{\mathbf{C}_L}\cap\xi^0)\leq \const{5}(8\lceil\sqrt{d}\rceil+2)^d L^d$}' where:
\[\overline{\mathbf{C}_L}=\left[-(4\lceil\sqrt{d}\rceil+1) L,(4\lceil\sqrt{d}\rceil+1) L\right]^d.\] 

On $\mathcal{A}_L$, since each $C^L_j$ contains a point of $\xi^0$, points on the boundary of $\mathbf{C}_L$ are within  distance at most $\sqrt{d}L$ from the nuclei of there respective Voronoi cells. Note that $0\in\xi^0$ can not be one of these nuclei. It follows that the union of these Voronoi cells separates $0$ (which is in its own cell) from the exterior of the union of the cells intersecting $\mathbf{C}_L$. Thus, on $\mathcal{A}_L$, we have:
\[\deg_{\DT(\xi^0)}(0)\leq \#\left(\overline{\mathbf{C}_L}\cap\xi^0\right)\leq \const{5}(8\lceil\sqrt{d}\rceil+2)^d L^d.\]

Hence, for $\const{5}(8\lceil\sqrt{d}\rceil+2)^d L^d\leq D <\const{5}(8\lceil\sqrt{d}\rceil+2)^d (L+1)^d$, {\bf (V')}, {\bf (D')} and the union bound imply that:
\begin{align*}
\P_0\left[\deg_{\DT(\xi^0)}(0)\geq D\right]&\leq \P_0\left[\deg_{\DT(\xi^0)}(0)\geq \const{5}(8\lceil\sqrt{d}\rceil+2)^d L^d\right]\\&\leq \P_0\left[\mathcal{A}_L^c\right]\\
&\leq \sum_{j\in J_L}  \P_0\left[\#\left(C^L_j\cap\xi^0\right)=0\right]\\
&\qquad\qquad+ \P_0\left[\#\left(\overline{\mathbf{C}_L}\cap\xi^0\right)\geq \const{5}(8\lceil\sqrt{d}\rceil+2)^d L^d\right]\\
&\leq O(L^{d-1})e^{-\const{4}L^d}+e^{-\const{6}(8\lceil\sqrt{d}\rceil+2)^d L^d}=O(e^{- \const{14}L^d}).
\end{align*}

Finally, for $k\in\NN^*$, we have:
\begin{align*}
\E_0\left[\left(\deg_{\DT(\xi^0)}(0)\right)^k\right]&\leq \sum_{D=1}^\infty D^k\P_0\left[\deg_{\DT(\xi^0)}(0)\geq D\right]\\
& \leq \sum_{D=1}^{\lfloor \const{5}(8\lceil\sqrt{d}\rceil+2)^d\rfloor} D^k\P_0\left[\deg_{\DT(\xi^0)}(0)\geq D\right]\\
&\qquad\qquad+\sum_{L=1}^\infty O(L^{d(k+1)-1})\P_0\left[\deg_{\DT(\xi^0)}(0)\geq \const{5}(8\lceil\sqrt{d}\rceil+2)^d L^d\right]\\
& \leq \sum_{D=1}^{\lfloor \const{5}(8\lceil\sqrt{d}\rceil+2)^d\rfloor} D^k+\sum_{L=1}^\infty O(L^{d(k+1)-1}e^{- \const{14}L^d})<\infty.
\end{align*}

\begin{center}
\begin{figure}
\includegraphics[scale=0.7]{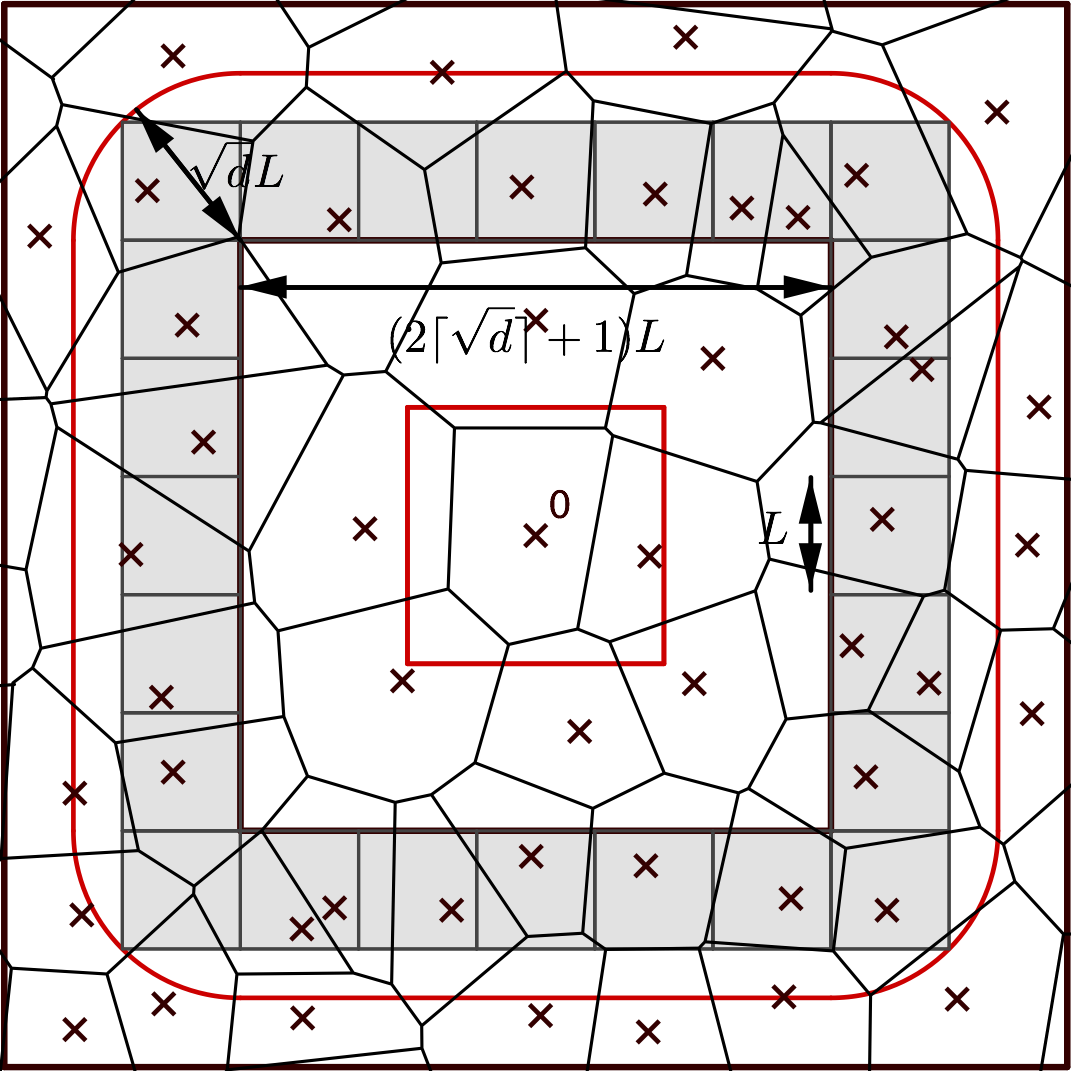}
\caption{\label{DessinDegre}If each (grey) cube of side $L$ that border $[-(2\lceil\sqrt{d}\rceil+1) L,(2\lceil\sqrt{d}\rceil+1) L]^d$ contains at least a point of $\xi^0$, then the nuclei of the Voronoi cells intersecting  the boundary of $[-(2\lceil\sqrt{d}\rceil+1) L,(2\lceil\sqrt{d}\rceil+1) L]^d$ lie in the region delimited by the red curves. In this case, the union of these cells surround $0\in\xi^0$ and the degree of $0$ in $\DT(\xi^0)$ is bounded by $\#([-(4\lceil\sqrt{d}\rceil+1) L,(4\lceil\sqrt{d}\rceil+1) L]^d\cap \xi^0)$.}
\end{figure}
\end{center}
Polynomial moments for $\max_{x\sim 0 \mbox{ in }\DT(\xi^0)}(0)$ can be proved to be finite similarly noticing that, on $\mathcal{A}_L$, $\max_{x\sim 0 \mbox{ in }\DT(\xi^0)}(0)$ is generously bounded by $\sqrt{d}(4\lceil\sqrt{d}\rceil+1)L$.
\end{proof}
\subsection[Good boxes and paths constructions]{Good boxes and paths constructions for applications of Proposition \ref{PropNonDege}}\label{SectGoodBoxes3}
Due to the sub-graph relation $\operatorname{G_n}(\xi)\subset \Gab(\xi)\subset \DT (\xi)$, it is enough to construct good boxes and paths for the creek-crossing graphs $\operatorname{G_n}, \,\operatorname{n}\geq 2$ to apply Proposition \ref{PropNonDege} and obtain Theorem \ref{ThprincAIP}. Note that the constructions can be easily made by hand in the case of Delaunay triangulations as in the Chapter 2 of \cite{Phd}.

We will use \cite[Lemma 8]{HNGS} to verify assumptions of Proposition \ref{PropNonDege}. 

\begin{lemm}[Lemma 8 in \cite{HNGS}]\label{LemmeHNGS}
Fix $\operatorname{n}\geq 2$. Assume that $\xi$ is distributed according to a simple, stationary, isotropic and $m-$dependent point process satisfying {\bf (V)} and {\bf (D)}. Assume in addition that its second factorial moment measure is absolutely continuous with respect to $2d-$Lebesgue measure and its density is bounded from above by some positive constant.

There exist then $0<\const{15},\const{16}<\infty$ such that:
\begin{equation*}
\P\big[\mathcal{A}_{\operatorname{G_n},x,L}\big]\geq 1-\const{15}\exp \big(-L^{\const{16}}\big),
\end{equation*}
where $\mathcal{A}_{\operatorname{G_n},x, L}$ is the event \emph{`any two vertices of $\operatorname{G_n}$ $u,v \in x+[-\frac{L}{2}, \frac{L}{2}]^d$ are connected by a $\operatorname{G_n}$-path contained in $x+[-L,L]^d$'}.
\end{lemm}

In the sequel, we assume that the underlying point process satisfies the hypothesis of Lemma \ref{LemmeHNGS}.

\subsubsection{Definition of the good boxes}
For $K\geq 1$ to be determined later, consider a partition of $\RR^d$ into boxes of side $K$:
\[B_\z=B^K_\z:=K\z+\Big[-\frac{K}{2},\frac{K}{2}\Big]^d,\,\z\in \ZZ^d.\]
We say that a box $B_{\z_1}$ is \emph{nice} if:
\begin{itemize}
\item[{\it -i-}] $\# \big(B_{\z_1}\cap\xi\big)\leq \const{2}K^d$,
\item[{\it -ii-}] when $B_{\z_1}$ is (regularly) cut into $\alpha_d^d:=(4\lceil\sqrt{d}\rceil+5)^d$ sub-boxes $b_i^{\z_1}$ of side $s:=K/\alpha_d$, each of these sub-boxes contains at least one point of $\xi$.
\end{itemize}

For $B_{\z_1}$ and $B_{\z_2}$ two neighboring boxes, we write $b^{\z_1,\z_2}_i$, $i=1,\dots,\alpha_d+1$, for the sub-boxes of side $s$ intersecting the line segment $[K\z_1,K\z_2]$. We denote by $c^{\z_1,\z_2}_i$ the center of the small box $b^{\z_1,\z_2}_i$ and by $v^{\z_1,\z_2}_i$ the nucleus of the Voronoi cell of $c^{\z_1,\z_2}_i$. 

We finally say that the box $B_{\z_1}$ is \emph{good} for $\operatorname{G_n}$ if it is nice and:
\begin{itemize}
\item[{\it -iii-}] for any ${\z_2}\in \ZZ^d$ with $\Vert {\z_1}-{\z_2}\Vert_\infty=1$ and any $i\in\{1,\dots, 2\lceil\sqrt{d}\rceil+3\}$, $\mathcal{A}_{\operatorname{G_n},c^{{\z_1},{\z_2}}_i,(1+\sqrt{d})s}$ holds.
\end{itemize}
 
\subsubsection{Construction of paths}
For each good box $B_{\z_i}$, we choose the nucleus of the Voronoi cell of $K\z_i$ as reference vertex and denote it by  $v_{\z_i}$. Thanks to {\it -ii-} in the definition of good boxes, one can check that $v_{\z_i}\in B_{\z_i}$.

Consider two neighboring good boxes $B_{\z_1}$ and $B_{\z_2}$. We want to verify that there exists a $\operatorname{G_n}(\xi)$-path between $v_{\z_1}$ and $v_{\z_2}$ so that condition (\ref{CondNonDege2}) of Proposition \ref{PropNonDege} holds. To this end, let us consider the collection $v^{\z_1,\z_2}_1=v_{\z_1},\dots ,v^{\z_1,\z_2}_{\alpha_d+1}=v_{\z_2}$ of the nuclei of the Voronoi cells of the centers of the sub-boxes $b^{\z_1,\z_2}_i$, $1\leq i\leq\alpha_d+1$.  We claim that, for $i=1,\dots, \alpha_d$, there is a $\operatorname{G_n}(\xi)$-path $\gamma^{\z_1,\z_2}_i$ between $v^{\z_1,\z_2}_i$ and $v^{\z_1,\z_2}_{i+1}$ included in $B_{\z_1}\cup B_{\z_2}$. Condition {\it -ii-} in the definition of nice boxes implies that $v^{\z_1,\z_2}_i$ and $v^{\z_1,\z_2}_{i+1}$ belong to $c^{\z_1,\z_2}_i+[-(1+\sqrt{d})s,(1+\sqrt{d})s]^d$. The existence of $\gamma^{\z_1,\z_2}_i$ then follows from {\it -iii-} in the definition of good boxes and the fact that $c^{\z_1,\z_2}_i+[-(1+\sqrt{d})s,(1+\sqrt{d})s]^d\subset B_{\z_1}\cup B_{\z_2}$. Finally, we obtain the path between $v_{\z_1}$ and $v_{\z_2}$ needed in Proposition \ref{PropNonDege} by concatenating the $\gamma_i$s and removing any loop.
This path has length at most $2\const{2}K^d$ since it is simple and contained in the union of two good boxes. 
\subsubsection{Domination}
Since the underlying point process $\xi$ is supposed $m$-dependent, the process of good boxes has a finite range of dependence $k$. Clearly, the fact that a box $B_{\z_1}$ is nice or not depends only on the behavior of $\xi$ in this box. Moreover, due to the definition of the edge set of $\operatorname{G_n}(\xi)$, the fact that two vertices $u,v\in \xi$ are connected by an edge in $\operatorname{G_n}(\xi)$ only depends on $B(u,n\Vert u-v\Vert)\cap\xi$. Hence, in order to decide if $\mathcal{A}_{\operatorname{G_n},c^{{\z_1},{\z_2}}_i,(1+\sqrt{d})s}$ holds or not, we only need to know $B\big(c^{{\z_1},{\z_2}}_i,(2n+1)\sqrt{d}(1+\sqrt{d})s\big)\cap\xi$.
Thanks to \cite[Theorem 0.0]{Liggett97}, we know that there exists $p^*=p^*(d,k)<1$ such that if $\P\big[B_\z\mbox{ is good}\big]\geq p^*$, then the process of the good boxes $\mathbb{X}=\{X_\z, \z\in\ZZ^d\}$ dominates an independent site percolation process on $\ZZ^d$ with parameter greater than $p^\text{site}_c(\ZZ^2)$. We will chose $s$ and thus $K$ such that $\P\big[B_\z\mbox{ is good}\big]\geq p^*$.

Using {\bf (V)} and {\bf (D)}, one can show that, for $s$ large enough:
\begin{align*}
\P\big[B_\z\mbox{ is not nice}\big]&\leq \frac{1-p^*}{2}. 
\end{align*}

Moreover, using stationarity and isotropy of $\xi$ and Lemma \ref{LemmeHNGS}:
\begin{align*}
\P\big[\exists \z'\sim \z,\exists i\in&\{1,\dots, 2\lceil\sqrt{d}\rceil+3\}\mbox{ s.t. }\mathcal{A}_{\operatorname{G_n},c^{{\z},{\z'}}_i,(1+\sqrt{d})s}\mbox{ does not hold}\big]\\
&\leq\sum_{\z'\sim \z}\sum_{i=1}^{2\lceil\sqrt{d}\rceil+3}\P\big[(\mathcal{A}_{\operatorname{G_n},c^{{\z_1},{\z_2}}_i,(1+\sqrt{d})s})^c\big]\\
&=2d(2\lceil\sqrt{d}\rceil+3)\P\big[(\mathcal{A}_{\operatorname{G_n},0,(1+\sqrt{d})s})^c\big]\\
&=2d(2\lceil\sqrt{d}\rceil+3)\const{15}\exp \big(-(1+\sqrt{d})^{\const{16}}s^{\const{16}}\big)\\
&\leq \frac{1-p^*}{2}, 
\end{align*}
for $s$ large enough. We have just verified that the hypotheses listed in Subsection \ref{condprocAIP} allow us to apply Propositions \ref{PropCV} and \ref{PropNonDege} in the cases of Delaunay triangulations, Gabriel graphs and \emph{creek-crossing} graphs. Theorem \ref{ThprincAIP} thus follows.
\appendix
\section{Technical results for the proof of Lemma \ref{LemmeComparaisonDesCoefficientsDeDiffusion}}
In this appendix we collect preliminary results used in the proof of Lemma \ref{LemmeComparaisonDesCoefficientsDeDiffusion} and the assumptions of Proposition \ref{PropNonDege} are supposed to hold.

\subsection{Stationary point processes and their Palm versions}
We first recall useful facts concerning stationary point processes and their Palm versions.
\begin{lemm}\label{LemmeUtilPoisson}
\begin{enumerate}
\item\label{LemmeUtilPoisson0} For $p\geq 1$, if $A$ is the disjoint union of $k$ unit cubes, then
\begin{equation*}
\E \big[\#\big(\xi\cap A\big)^p\big]\leq k^p \E \big[\#\big(\xi\cap [-1/2,1/2]^d\big)^p\big].
\end{equation*} 
\item\label{LemmeUtilPoissonPre1} Let $1\leq p\leq 8$. Let $(A_n)_{n\in\NN}$ be a convex averaging sequence of bounded Borel sets in $\RR^d$, that is a sequence of convex sets satisfying $A_n\subset A_{n+1}$ and $r(A_n)\longrightarrow\infty$ where $r(A):=\sup\{r:\,A\mbox{ contains a ball of radius }r\}$. The following ergodic result holds:
\begin{equation*}
\frac{\#\big(\xi\cap A_n\big)}{\lambda \Vol(A_n)}\xrightarrow[]{ \quad a.s.,\,\,L^p(\N,\P)\quad}1,
\end{equation*} 
where $\lambda$ denotes the intensity of the point process.
\item\label{LemmeUtilPoisson1} Let $1\leq p\leq 8$. For any $C\in \RR$:
\begin{equation*}
\frac{\lambda \Vol(C_{2N})}{\#\big(\xi\cap C_{2N}\big)+CN^{d-1}}\xrightarrow[]{\quad L^p(\N,\P)\quad }1.
\end{equation*} 
\item\label{LemmeUtilPoisson2} For any  $h\in L^2(\Nz,\P_0)$:
\begin{equation*}
\E\Big[\int_A h(\tau_x\xi)\xi(\d x)\Big]=\lambda \Vol(A) \E_0[h],\,\forall A\in\mathcal{B}(\RR^d).
\end{equation*} 
\item\label{LemmeUtilPoisson3} For any $h\in L^2(\Nz,\P_0)$:
\begin{equation*}
\lim_{N\rightarrow\infty}\E\big[\gm_N^\xi\big( h(\tau_x\xi)\big)\big]=\E_0[h],
\end{equation*} 
with $\gm_N^\xi$ as defined in Subsection \ref{PeriodizedMedium}.
\item\label{LemmeUtilPoisson4} Fix $0<\alpha<1$ and define $M=M(N,\alpha):=N-\lfloor N^\alpha\rfloor$. For any $h\in L^2(\Nz,\P_0)$:
\begin{equation*}
\lim_{N\rightarrow\infty}\E\Bigg[\frac{1}{\# (\xi\cap C_{2N})+\#(\Gamma_N)}\int_{C_{2M}} h(\tau_x\xi)\xi(\d x)\Bigg]=\E_0[h].
\end{equation*}
\item\label{VimpliesV'} If the stationary point process satisfies {\bf (V)} and $\E[\#(\xi\cap [-1/2,1/2]^d)^2]<\infty$, then its Palm version satisfies {\bf (V')}.
\item\label{DimpliesD'} If the stationary point process satisfies {\bf (D)} and $\E[\#(\xi\cap [-1/2,1/2]^d)^2]<\infty$, then its Palm version satisfies {\bf (D')}.
\end{enumerate}
\end{lemm} 
\begin{dem} Since the point process is stationary, (\ref{LemmeUtilPoisson0}) is immediate. Point (\ref{LemmeUtilPoissonPre1}) is a consequence of \cite[Proposition 12.2.II]{DVJvol2} and of the ergodicity of $\P$. Observe that:
\begin{align*}
\E\Bigg[\Bigg\vert \frac{\lambda\Vol (C_{2N})}{\#\big(\xi\cap C_{2N}\big)+CN^{d-1}}&-1\Bigg\vert^p\Bigg]\\
&\leq\Big(\frac{\lambda N}{C}\Big)^p\P\Big[\#\big(\xi\cap C_{2N}\big)\leq \const{7} N^d\Big]\\
&\qquad+ \E\Bigg[\Bigg\vert \frac{\lambda\Vol (C_{2N})}{\#\big(\xi\cap C_{2N}\big)+CN^{d-1}}-1\Bigg\vert^p\mathbf{1}_{\#(\xi\cap C_{2N})\geq \const{7} N^d}\Bigg].
\end{align*}
Thanks to assumption (\ref{ConProcCV}) in Proposition \ref{PropNonDege}, the first summand vanishes in the limit. Using point (\ref{LemmeUtilPoissonPre1}), one has:
\[ f_N(\xi):=\Bigg\vert \frac{\lambda\Vol (C_{2N})}{\#\big(\xi\cap C_{2N}\big)+CN^{d-1}}-1\Bigg\vert^p\mathbf{1}_{\#(\xi\cap C_{2N})\geq \const{7} N^d}\xrightarrow[\, N\rightarrow \infty\,]{\quad a.s. \quad}0.\]
Since $\vert f_N\vert$ is bounded, the second summand goes to 0 by the dominated convergence theorem and point (\ref{LemmeUtilPoisson1}) is proved.  Point (\ref{LemmeUtilPoisson2}) is a special case of Campbell formula (see for example \cite[Theorem 3.3.3]{SW}). Thanks to point (\ref{LemmeUtilPoisson2}), in order to derive point (\ref{LemmeUtilPoisson3}), it suffices to show that:
\begin{equation*}
\lim_{N\rightarrow\infty}\E\Bigg[\Big(\frac{1}{\# (\xi\cap C_{2N})+\#(\Gamma_N)}-\frac{1}{\lambda\Vol(C_{2N-2r_c})}\Big)\int_{C_{2N-2r_c}} h(\tau_x\xi)\xi(\d x)\Bigg]=0.
\end{equation*}
Using four times the Cauchy-Schwarz inequality, and points \emph{\ref{LemmeUtilPoisson0}.} and \emph{\ref{LemmeUtilPoisson2}.}, one has:
\begin{align*}
\E\Bigg[\Big(&\frac{1}{\# (\xi\cap C_{2N})+\#(\Gamma_N)}-\frac{1}{\lambda\Vol(C_{2N-2r_c})}\Big)\int_{C_{2N-2r_c}} h(\tau_x\xi)\xi(\d x)\Bigg]^2\nonumber\\
\leq& \E\Bigg[\Big(\frac{\lambda\Vol(C_{2N-2r_c})}{\# (\xi\cap C_{2N})+\#(\Gamma_N)}-1\Big)^2\frac{\#(\xi \cap C_{2N-2r_c})}{\lambda\Vol (C_{2N-2r_c})}\Bigg]\nonumber\\
&\times\frac{1}{\lambda\Vol(C_{2N-2r_c})}\E\Bigg[\frac{1}{\#(\xi \cap C_{2N-2r_c})}\Big(\int_{C_{2N-2r_c}} h(\tau_x\xi)\xi(\d x)\Big)^2\Bigg]\nonumber\\
\leq& \E\Bigg[\Big(\frac{\lambda\Vol(C_{2N-2r_c})}{\# (\xi\cap C_{2N})+\#(\Gamma_N)}-1\Big)^2\frac{\#(\xi \cap C_{2N-2r_c})}{\lambda\Vol (C_{2N-2r_c})}\Bigg]
\\&\times \frac{1}{\lambda\Vol(C_{2N-2r_c})}\E\Bigg[\int_{C_{2N-2r_c}} h(\tau_x\xi)^2\xi(\d x)\Bigg]\nonumber\\
=& \E\Bigg[\Big(\frac{\lambda\Vol(C_{2N-2r_c})}{\# (\xi\cap C_{2N})+\#(\Gamma_N)}-1\Big)^2\frac{\#(\xi \cap C_{2N-2r_c})}{\lambda\Vol (C_{2N-2r_c})}\Bigg]\E_0\big[ h^2\big]\nonumber\\
\leq& \frac{1}{\lambda}\E\Bigg[\Big(\frac{\lambda\Vol(C_{2N-2r_c})}{\# (\xi\cap C_{2N})+\#(\Gamma_N)}-1\Big)^4\Bigg]^\frac{1}{2}\E\Bigg[\frac{\#(\xi \cap C_{2N-2r_c})^2}{\Vol (C_{2N-2r_c})^2}\Bigg]^\frac{1}{2}\E_0\big[ h^2\big]\nonumber\\
\leq& \frac{1}{\lambda}\E\Bigg[\Big(\frac{\lambda\Vol(C_{2N-2r_c})}{\# (\xi\cap C_{2N})+\#(\Gamma_N)}-1\Big)^4\Bigg]^\frac{1}{2}\E\big[\#(\xi \cap [-1/2,1/2]^d)^2\big]^\frac{1}{2}\E_0\big[ h^2\big].
\end{align*}
This completes the proof since $h\in L^2(\Nz,\P_0)$, $\E\big[\#(\xi \cap [-1/2,1/2]^d)^2\big]<\infty$ and by point (\ref{LemmeUtilPoisson1}):
\[\E\Bigg[\Big(\frac{\lambda\Vol(C_{2N-2r_c})}{\# (\xi\cap C_{2N})+\#(\Gamma_N)}-1\Big)^4\Bigg]\xrightarrow[\,N\longrightarrow\infty\,]{}0.\] 
The proof of point (\ref{LemmeUtilPoisson4}) is very similar to the previous one.

Let us show that {\bf (V)} and $\E[\#(\xi\cap [-1/2,1/2]^d)^2]<\infty$ imply {\bf (V')}. By the definition of the Palm measure, we have:
\[\P_0\left[\#\left(\xiz\cap \left(x+[-L/2,L/2]^d\right)\right)=0\right]=\frac{1}{\lambda}\E\left[\sum_{y\in\xi\cap [-1/2,1/2]^d}\mathbf{1}_{\#(\tau_y\xi\cap(x+[-L/2,L/2]^d))=0}\right],\]
where $\lambda$ denotes the intensity of the point process. Let us cut $x+[-L/2,L/2]^d$ into $2^d$ cubes of side $L/2$ having $x$ at a corner $C^1_{L/2},\dots,C^{2^d}_{L/2}$. If the sum in the last expression is not 0 and $L\geq 1$, then at least one of the $C^j_{L/2}$s is empty. Hence, the union bound and the Cauchy-Schwarz inequality imply that
\begin{align*}
\P_0\left[\#\left(\xiz\cap \left(x+[-L/2,L/2]^d\right)\right)=0\right]&\leq\frac{1}{\lambda}\E\left[\#\left(\xi\cap [-1/2,1/2]^d\right)
\mathbf{1}_{\#(\xi\cap (x+C^j_{L/2})=0\text{ for some }j\in\{1,\dots,2^d\}}
\right]\\
&\leq\frac{2^\frac{d}{2}}{\lambda}\E\left[\#\left(\xi\cap [-1/2,1/2]^d\right)^2\right]^\frac{1}{2}\P\left[\#(\xi\cap [-L/4,L/4]^d)=0\right]^\frac{1}{2}.
\end{align*}
Point (\ref{VimpliesV'}) then follows from {\bf (V)}. Similarly, we can write that
\begin{align*}
\P_0\left[\#\left(\xiz\cap \left(x+[-L/2,L/2]^d\right)\right)\geq \const{5}L^d\right]&=\frac{1}{\lambda}\E\left[\sum_{y\in\xi\cap [-1/2,1/2]^d}\mathbf{1}_{\#(\tau_y\xi\cap(x+[-L/2,L/2]^d))\geq \const{5}L^d}\right]\\
&\leq\frac{1}{\lambda}\E\left[\#\left(\xi\cap [-1/2,1/2]^d\right)
\mathbf{1}_{\#(\xi\cap (x+[-L,L]^d))\geq\const{5}L^d}
\right]\\
&\leq\frac{1}{\lambda}\E\left[\#\left(\xi\cap [-1/2,1/2]^d\right)^2\right]^\frac{1}{2}\P\left[\#(\xi\cap [-L,L]^d)\geq \const{5}L^d\right]^\frac{1}{2}.
\end{align*}
This proves point (\ref{DimpliesD'}) thanks to {\bf (D)} by choosing $\const{5}=2^d\const{2}$.
\end{dem}
\subsection{The mean boundary contributions are negligible}
\begin{lemm}\label{LemmMeanBContribNegl}
Let $h_N^\xi:\gV_N^\xi\longmapsto \RR$ be a function verifying:
\begin{equation}\label{EqComparaison4}
\big\vert h_N^\xi (x)\big\vert\leq \const{17}\left\lbrace
\begin{array}{ll}
\#\big(\xi \cap \overline{B_2(x,r_c)}\big)&\mbox{if }x\in\gQ_N^\xi\\
\dfrac{\#\big(B_N^\xi\big)}{\#\big(\Gamma_N\big)}&\mbox{if }x\in\Gamma_N
\end{array}
\right. ,
\end{equation}
for some $\const{17}<\infty$ independent of $N$. 

Then the mean boundary contributions vanish in the limit:
\begin{equation}\label{EqComparaison5}
\lim_{N\rightarrow\infty}\E \Bigg[\frac{1}{\#\big(\xi \cap C_{2N}\big)+\#\Gamma_N}\sum_{x\in\gV_N^\xi\setminus\gQ^\xi_{N-r_c}}\big\vert h_N^\xi (x)\big\vert^p\Bigg]=0,\quad 1\leq p\leq 4.
\end{equation}

In particular, this is the case for $h_N^\xi=\varphi_N^\xi$ or $\psi_N^\xi$ defined in (\ref{EqLemmPeriodizedMedium2}). 
\end{lemm}
\begin{dem}
Note that with (\ref{EqComparaison4}):
\begin{equation*}
\sum_{x\in\gV_N^\xi\setminus\gQ^\xi_{N-r_c}}\big\vert h_N^\xi (x)\big\vert^p\leq 
\const{17}^p\Bigg\{\#\Gamma_N\frac{\#\big(B_N^\xi\big)^p}{\#\big(\Gamma_N\big)^p}+ \sum_{x\in\gQ_N^\xi\setminus\gQ^\xi_{N-r_c}}\#\big(\xi\cap \overline{B_2(x,r_c)}\big)^p\Bigg\}.
\end{equation*}
But, on the one hand:
\begin{align}\label{EqComparaison7}
\E\Bigg[\frac{\#\Gamma_N}{\#\big(\xi \cap C_{2N}\big)+\#\Gamma_N}\frac{\#\big(B_N^\xi\big)^p}{\#\big(\Gamma_N\big)^p}\Bigg]&\leq\E\Bigg[\frac{\#(\Gamma_N)^2}{\big(\#\big(\xi \cap C_{2N}\big)+\#\Gamma_N\big)^2}\Bigg]^\frac{1}{2}\E\Bigg[\frac{\#\big(B_N^\xi\big)^{2p}}{\#\big(\Gamma_N\big)^{2p}}\Bigg]^\frac{1}{2}\nonumber\\
&\leq \const{18}\E\Bigg[\frac{\#\Gamma_N^2}{\big(\#\big(\xi \cap C_{2N}\big)+\#\Gamma_N\big)^2}\Bigg]^\frac{1}{2},
\end{align}
where $\const{18}$ is a constant which does not depend on $N$ thanks to the point (\ref{LemmeUtilPoisson0}) of Lemma \ref{LemmeUtilPoisson}. On the other hand, using the Cauchy-Schwarz inequality and Lemma \ref{LemmeUtilPoisson} (\ref{LemmeUtilPoisson2}):
\begin{align}\label{EqComparaison8}
\E \Bigg[&\frac{1}{\#\big(\xi \cap C_{2N}\big)+\#\Gamma_N}\sum_{x\in\gQ_N^\xi\setminus\gQ^\xi_{N-r_c}}\#\big(\xi\cap \overline{B_2(x,r_c)}\big)^p\Bigg]\nonumber\\
&\quad\leq \E \Bigg[\frac{\# \big(\xi\cap C_{2N}\setminus C_{2N-2r_c}\big)}{\big(\#\big(\xi \cap C_{2N}\big)+\#\Gamma_N\big)^2}\Bigg]^\frac{1}{2}\nonumber\\&\quad\quad\quad\times\E \Bigg[\frac{1}{\# \big(\xi\cap C_{2N}\setminus C_{2N-2r_c}\big)}\Bigg(\sum_{x\in\gQ_N^\xi\setminus\gQ^\xi_{N-r_c}}\#\big(\xi\cap \overline{B_2(x,r_c)}\big)^p\Bigg)^2\Bigg]^\frac{1}{2}\nonumber\\
& \quad\leq \left(\lambda\Vol \big(C_{2N}\setminus C_{2N-2r_c}\big)\right)^\frac{1}{2}\E \Bigg[\frac{\# \big(\xi\cap C_{2N}\setminus C_{2N-2r_c}\big)}{\big(\#\big(\xi \cap C_{2N}\big)+\#\Gamma_N\big)^2}\Bigg]^\frac{1}{2}\E_0 \Big[\#\big( \xi\cap B_2(0,r_c)\big)^{2p}\Big]^\frac{1}{2}\nonumber\\
& \quad\leq \Bigg(\frac{\Vol \big(C_{2N}\setminus C_{2N-2r_c}\big)}{\Vol \big(C_{2N}\big)}\Bigg)^\frac{1}{2}\E \Bigg[\frac{\lambda\Vol \big(C_{2N}\big)}{\#\big(\xi \cap C_{2N}\big)+\#\Gamma_N}\Bigg]^\frac{1}{2}\E_0 \Big[\#\big( \xi\cap B_2(0,r_c)\big)^{2p}\Big]^\frac{1}{2}.
\end{align}
Hence, (\ref{EqComparaison5}) is proved since the r.h.s.\@\xspace of (\ref{EqComparaison7}) and the r.h.s of (\ref{EqComparaison8}) go to 0 thanks to Lemma \ref{LemmeUtilPoisson} (\ref{LemmeUtilPoisson1}).

Let us check that $\varphi_N^\xi$ defined in (\ref{EqLemmPeriodizedMedium2}) satisfies (\ref{EqComparaison4}).
If $x\in\gQ_N^\xi$, then:
\begin{align*}
\big\vert\varphi_N^\xi(x) \big\vert \leq&\sum_{y\in\gQ_N^\xi}\mathbf{1}_{\{x,y\}\in E_{G(\xi)}, \Vert x-y\Vert\leq r_c}\big\vert y_1-x_1\big\vert\\
&+\frac{1}{\#\big(\Gamma_N\big)}\sum_{y\in\Gamma_N}\big(\mathbf{1}_{x\in B_N^{\xi,-}}\big\vert x_1+N\big\vert +\mathbf{1}_{x\in B_N^{\xi,+}}\big\vert x_1-N\big\vert\big)\\
\leq & \big(\#\big(\xi \cap \overline{B_2(x,r_c)}\big)-1\big)r_c+r_c=\#\big(\xi \cap \overline{B_2(x,r_c)}\big)r_c.
\end{align*}
If $x\in\Gamma_N$, then:
\begin{equation*}
\big\vert \varphi_N^\xi(x)\big\vert \leq\frac{1}{\#\big(\Gamma_N\big)}\sum_{y\in B_N^{\xi,-}}\big\vert y_1+N\big\vert+\frac{1}{\#\big(\Gamma_N\big)}\sum_{y\in B_N^{\xi,+}}\big\vert y_1-N\big\vert \leq \frac{r_c\#\big(B_N^\xi\big)}{\#\big(\Gamma_N\big)}.
\end{equation*}
Hence, $\varphi_N^\xi$ satisfies (\ref{EqComparaison4}) with $\const{17}=r_c$. The same computations show that $\psi_N^\xi$ satisfies (\ref{EqComparaison4}) with $\const{17}=r_c^2$.
\end{dem}
\subsection{Control of exit times}\label{SubsectPrLemmeHittingTimes}
 Let us give a sketch of proof of the following result which is a slight rewriting of \cite[Lemma 5]{FSS}.
\begin{lemm}\label{LemmeHittingTimes}
Let $\mathcal{T}_N^\xi$ and $\mathcal{T}_{N,x}$ be as in (\ref{DefHittingTimes}) and (\ref{DefHittingTimesEnvironment}), $0<\alpha<1$ and $M:=N-\lfloor N^\alpha\rfloor$. Then,
\begin{equation}\label{EqLemmeHittinTimes1}
\lim_{N\rightarrow\infty}\E\Big[\gm_N^\xi\big(\mathbf{1}_{x\in C_{2M}}\bP_{N,x}^\xi\big[\mathcal{T}_N^\xi\leq t\big]\big)\Big]=0,
\end{equation}
\begin{equation}\label{EqLemmeHittinTimes2}
\lim_{N\rightarrow\infty}\E\Big[\gm_N^\xi\big(\mathbf{1}_{x\in C_{2M}}\widehat{\bP}_{\tau_x\xi}\big[\mathcal{T}_{N,x}\leq t\big]\big)\Big]=0.
\end{equation}
\end{lemm}
\begin{dem}
Recall the notations from Subsections \ref{CutOff} and \ref{PeriodizedMedium}.
We first note that the expectations in (\ref{EqLemmeHittinTimes1}) and (\ref{EqLemmeHittinTimes2}) coincide. Actually, for each $N\in\NN^*$, each locally finite set $\xi\subset\RR^d$ and each $x\in\xi$, one can define a probability measure $\mu$ on $\Omega_N^\xi\times\Xi^0$ such that:
\[\mu(A\times \Xi^0)=\bP_{N,x}^\xiz (A), \quad \forall A\in\mathcal{B}(\Omega_N^\xi),\]
\[\mu(\Omega_N^\xi\times B)=\widehat{\bP}_{\tau_x\xi} (B), \quad \forall B\in\mathcal{B}(\Xi^0),\]
and such that, $\mu$-a.s.\@\xspace, $\mathcal{T}_N^\xi=\mathcal{T}_{N,x}$ and $\omega_t=x+X_t(\hat{\underline{\xiz}})$, $0\leq t<\mathcal{T}_N^\xi$. The existence of such a coupling implies that $\bP_{N,x}^\xi[\mathcal{T}_N^\xi\leq t]=\widehat{\bP}_{\tau_x\xi}[\mathcal{T}_{N,x}\leq t]$ and we only need to prove (\ref{EqLemmeHittinTimes1}).

To this end, let us part $C_{2N-2r_c}\setminus C_{2M}$ into unit cubes $Q_i=Q_{N,i},\,i\in I=I_N$. For $n\in \NN^*$, we set $I_n^*:=\{\underline{l}:=(l_1,\dots,l_n)\in I^n:\,l_j\neq l_k,\,j\neq k\}$. For $\underline{\omega}\in \Omega_N^\xi$ such that $\mathcal{T}_N^\xi(\underline{\omega})<\infty$, let us define:
\[k=k(\underline{\omega}):=\#\big\{i\in I:\,\exists\, 0\leq t<\mathcal{T}_N^\xi(\underline{\omega}),\,\omega_t\in Q_i\big\},\]
and $\underline{i}:=(i_1,\dots\,i_k)\in I_k^*$ the list of such indices sorted in the chronological order of first visits. For $j=1,\dots, k$, we denote by $x_j$ the first site of $Q_j$ visited by $\underline{\omega}$ and by $t_j$ the time spent at $x_j$ during the first visit. 

Now let  $\{T_i^{\xi},\,i\in I,\,\xi\in\N\}$ be a family of independent exponential random variables with respective parameters
 $\#\big(\xi\cap \widetilde{Q}_i\big)$, where
\[\widetilde{Q}_i:=\{y \in \RR^d :\, \operatorname{dist}(y,Q_i) \leq r_c\}.\]
Remark that given $\xi$, $k$  and $\underline{x}=(x_1, \dots, x_k)$, $t_j $ ($1\leq j \leq k$) are independent exponential variables  with respective parameters no larger than $\#\big(\xi\cap \widetilde{Q}_i\big)$. Thus, $n,\,\underline{y}$ being fixed:
\begin{equation}\label{EqComparaison14}
\bP_{N,x}^\xi\big[\mathcal{T}_N^\xi\leq t\,\big\vert\,k=n, \underline{x}=\underline{y}\big]\leq \PP\big[T_{l_1}^\xi +\dots +T_{l_n}^\xi\leq t\big].
\end{equation}
Since $k\geq k_{\min}:= \lfloor N^\alpha\rfloor-r_c $, one has using (\ref{EqComparaison14}):
\begin{align}\label{EqComparaison15}
\E\Big[\gm_N^\xi&\big(\mathbf{1}_{x\in C_{2M}}\bP_{N,x}^\xi\big[\mathcal{T}_N^\xi\leq t\big]\big)\Big]\nonumber\\
&= \sum_{n=\lfloor N^\alpha\rfloor-r_c}^{\#I}\sum_{\underline{l}\in I_n^*}\E\Big[\gm_N^\xi\big(\mathbf{1}_{x\in C_{2M}}\sum_{\underline{y}\in \prod_{j=1}^nQ_{l_j}\cap \gV_N^\xi}\bP_{N,x}^\xi\big[\mathcal{T}_N^\xi\leq t,\,k=n, \underline{x}=\underline{y}\big]\big)\Big]\nonumber\\
&\leq \sum_{n=\lfloor N^\alpha\rfloor-r_c}^{\#I}\sum_{\underline{l}\in I_n^*}\E\Big[\gm_N^\xi\big(\mathbf{1}_{x\in C_{2M}}\bP_{N,x}^\xi\big[k=n, \underline{i}=\underline{l}\big]\big)\times\PP\big[T_{l_1}^\xi +\dots +T_{l_n}^\xi\leq t\big]\Big].
\end{align}
It remains to bound the r.h.s of (\ref{EqComparaison15}). To this end, for $\kappa >0$ and $\underline{l}\in I_n^*$, let us define the event:
 \[\mathcal{A}=\mathcal{A}(\kappa,\underline{l})=\Big\{\xi\in\N:\,\#\big\{j:\,1\leq j \leq n,\,\# (\xi\cap \widetilde{Q}_j)>\kappa \E\big[\# (\xi\cap \widetilde{Q}_1)\big]\big\}>\frac{n}{2}\Big\}.\]

Due to the Markov inequality and the stationarity of $\P$, one has:
\begin{align*}
\P(\mathcal{A})&\leq \frac{2}{n}\E\Big[\#\big\{j:\,1\leq j \leq n,\,\# (\xi\cap \widetilde{Q}_j)>\kappa \E\big[\# (\xi\cap \widetilde{Q}_1)\big]\big\}\Big]\\
&\leq 2\P\Big[\# (\xi\cap \widetilde{Q}_1)>\kappa \E\big[\# (\xi\cap \widetilde{Q}_1)\big]\Big]\xrightarrow[\,\kappa\rightarrow\infty\,]{}0.
\end{align*}

Note that if $\xi\not\in\mathcal{A}$, at least $\lceil n/2\rceil$ of the exponential random variables $T_{l_1}^\xi,\dots, T_{l_n}^\xi$ have parameters smaller than $\kappa \E[\# (\xi\cap \widetilde{Q}_1)]$. It follows that on $\mathcal{A}^c$:
\[\PP\big[T_{l_1}^\xi +\dots +T_{l_n}^\xi\leq t\big]\leq \rho (n,\kappa) :=e^{-\kappa \E[\# (\xi\cap \widetilde{Q}_1)]t}\sum_{r=\lceil \frac{n}{2}\rceil}^\infty\frac{\big(\kappa \E\big[\# (\xi\cap \widetilde{Q}_1)\big]t\big)^r}{r!}.\]

Finally, we deduce by collecting bounds that:
\begin{align*}
\E\Big[\gm_N^\xi&\big(\mathbf{1}_{x\in C_{2M}}\bP_{N,x}^\xi\big[\mathcal{T}_N^\xi\leq t\big]\big)\Big]\nonumber\\
&\leq \E\Big[\gm_N^\xi\big(\mathbf{1}_{x\in C_{2M}} \sum_{n=\lfloor N^\alpha\rfloor-r_c}^{\#I}\sum_{\underline{l}\in I_n^*}\bP_{N,x}^\xi\big[k=n, \underline{i}=\underline{l}\big]\big) \big(\mathbf{1}_{\mathcal{A}}+\mathbf{1}_{\mathcal{A}^c}\rho (\lfloor N^\alpha\rfloor -r_c,\kappa)\big)\Big]\\
&\leq 2\P\Big[\# (\xi\cap \widetilde{Q}_1)>\kappa \E\big[\# (\xi\cap \widetilde{Q}_1)\big]\Big]+\rho (\lfloor N^\alpha\rfloor -r_c,\kappa).
\end{align*}

This gives the result by letting $N$ and then $\kappa$ go to the infinity.
\end{dem}
\section{Examples of point processes}\label{Expp}
In this appendix, we check that the assumptions given in Subsection \ref{condprocAIP} are satisfied by PPPs, MCPs and MHPs.
\subsection{Poisson point processes}
Stationarity, isotropy, ergodicity, aperiodicity and the finite range of dependence condition directly follow from the definition of homogeneous Poisson point processes. It is also known that stationary Poisson point processes are almost surely in general position. Standard computations give {\bf (V)}, {\bf (D)} and that $\#(\xi\cap [0,1]^d)$ admit polynomial moments of any order. The second factorial intensity measure of a PPP($\lambda$) admits $\lambda^2$ as density w.r.t.\@\xspace the $(2d)$-Lebesgue measure (see {\it e.g.\@\xspace} \cite[(2.30)]{ChiuSKM}).

\subsection{Mat\'ern cluster processes}
Mat\'ern cluster processes (MCPs) are particular cases of Neyman-Scott Poisson processes (see \cite[p. 171]{ChiuSKM}). Cluster processes are used as models for spatial phenomena, {\it e.g.\@\xspace} galaxy locations in space \cite{Galaxy} or epicenters of micro-earthquake locations \cite{VJEarthquakes}. 

MCPs are constructed as follows. One first chooses a PPP $Y$ of intensity $\lambda$ called the \emph{parent process}. For any $y\in Y$, a centered \emph{daughter process} $\xi_y$ is then chosen such that, given $Y$, $\{\xi_y\}_{y\in Y}$ are mutually independent PPP  with intensity $\mu$ in $B(0,R)$. Then, $\xi:=\bigcup_{y\in Y}(y+\xi_y)$ is a MCP with parameters $\lambda, \mu, R$. It is clear that such processes are simple, stationary and isotropic. Since its parent process has a finite range of dependence and daughter processes have bounded supports, any MCP has a finite range of dependence. MCPs can be seen as Cox processes (see \cite[p. 166]{ChiuSKM}). Their (diffusive) random intensity measures $\mu_Y$ have densities $Z^Y(x):=\sum_{y\in Y}\II_{B(y,R)}(x)$ w.r.t.\@\xspace the Lebesgue measure. In particular, $(d-1)$ dimensional hyperplanes and spheres are $\mu_Y-$null sets and MCPs are almost surely in general position and aperiodic. Assumptions {\bf (V)} and {\bf (D)} can be proved to hold as in \cite[\S 6.2]{RTCRWRG}. Standard computations also show that $\#(\xi\cap [0,1]^d)$ admit polynomial moments of any order. The second factorial moment measure of a MCP($\lambda, \mu, R$) is given, for $A_1,A_2\subset \RR^d$ two Borel sets, by:
\begin{align*}
\Lambda^{(2)}(A_1\times A_2)&=\E\left[\sum_{x_1\in\xi\cap A_1}\sum_{x_2\in\xi\cap A_2\setminus\{x_1\}}1\right]\\
&=\E\left[\E\left[\sum_{x_1\in\xi\cap A_1}\sum_{x_2\in\xi\cap A_2\setminus\{x_1\}}1\Bigg\vert Y\right]\right]\\
&=\E\left[\int_{A_1}\int_{A_2}Z^Y(x_1)Z^Y(x_2)\d x_1\d x_2\right]\\
&=\int_{A_1}\int_{A_2}\E\left[Z^Y(x_1)Z^Y(x_2)\right]\d x_1\d x_2,
\end{align*}
where we used the expression of the second factorial moment measure of a PPP driven by $Z^Y$ (see {\it e.g.\@\xspace} \cite[(2.28)]{ChiuSKM}) and the Fubini-Tonelli theorem.
Hence, it is absolutely continuous w.r.t.\@\xspace the $(2d)$-Lebesgue measure and admits
\[\E\left[Z^Y(x_1)Z^Y(x_2)\right]=\mu^2\left(\lambda^2\Vol \left(B(0,R)\right)^2+\lambda \Vol \left(B(x_1,R)\cap B(x_2,R)\right)\right)\]
as density.
\subsection{Mat\'ern hardcore processes}
In 1960, Mat\'ern introduced several hardcore models for point processes. These processes are dependent thinnings of PPPs and spread more regularly in space than PPPs. Such models are useful when competition for resources exists ({\it e.g.\@\xspace} tree or city locations, see \cite{Matern} and references therein). 

Given a realization $\xi$ of an initial  PPP of intensity $\lambda$ and $R>0$, let us define by:
\[\xi_{\text{I}}:=\big\{x\in\xi\,:\,\Vert x-y\Vert>R,  \forall y\in\xi\setminus\{x\}\big\}.\]
Then, $\xi_{\text{I}}$ is distributed according to a \emph{Mat\'ern I hardcore process} (MHP I). 

One can show that this point process satisfies {\bf (V)} in the same way as in \cite[\S 6.2]{RTCRWRG}. Its isotropy and stationarity follow from the construction. Since the MHP I is stochastically dominated by the initial PPP all the other hypotheses can be directly deduced from the case of Poisson point processes.

Given a realization $\xi$ of an initial  PPP of intensity $\lambda$, independent marks $\{T_x\}_{x\in\xi}$ uniformly distributed in $[0,1]$ and $R>0$, let us define:
\[\xi_{\text{II}}:=\big\{x\in\xi\,:T_x<T_y,  \forall y\in\xi\cap B(x,R)\big\}.\]  
Then, $\xi_{\text{II}}$ is distributed according to a \emph{Mat\'ern II hardcore process} (MHP II). One can check as above that MHP II satisfies the assumptions of Theorem \ref{ThprincAIP}. 

\subsection*{Acknowledgements}
The author thanks Jean-Baptiste Bardet and Pierre Calka, his PhD advisors, for introducing him to this subject and for helpful discussions, comments and suggestions. The author also thanks an anonymous referee for his/her careful reading and for his/her comments that significantly improved the paper.This work was partially supported by the French ANR grant PRESAGE (ANR-11-BS02-003) and the French research group GeoSto (CNRS-GDR3477).

\bibliography{biblio}
\bibliographystyle{alpha}
\end{document}